\documentclass[a4paper]{article}
\usepackage[all]{xy}
\usepackage{amsmath, amssymb, amsthm}
\usepackage{latexsym, color}
\usepackage[dvipdfmx, hiresbb]{graphicx}
\usepackage{float}

\newtheorem{thm}{Theorem}[section]
\newtheorem{prop}[thm]{Proposition}
\newtheorem{lem}[thm]{Lemma}
\newtheorem{defn}[thm]{Definition}

\newtheorem{cor}[thm]{Corollary}

\title{On Euclidean systems of ray classes}
\author{Yutaro Matsuno}

\newcommand{\n}{\normalfont}
\newcommand{\ds}{\displaystyle}
\newcommand{\p}{\mathfrak{p}}

\newcommand{\ded}{\mathfrak{o}}

\newcommand{\id}{\text{id}}
\newcommand{\m}{\mathfrak{m}}

\begin{document}

\maketitle

\begin{abstract}
Lenstra introduced the notion of Euclidean ideal classes, and Treatman extended it to Euclidean systems. In this paper, we formulate Euclidean systems for ray classes, and study their basic properties. In particular, we show that every Euclidean system of ray classes generates the corresponding ray class group. We further prove, assuming GRH, that if $K$ is a totally real Galois number field of degree $n\ge 3$ and $p$ is an odd rational prime which does not split completely in $K$, then for every $N>0$, every generating set of the ray class group $Cl_K^{(p)^N}$ with modulus $(p)^N$ is a Euclidean system.
\end{abstract}

\section{Introduction}

Euclidean domains are classical objects of interest in algebraic number theory. Every Euclidean domain is a PID, and hence the question of whether the ring of integers of a number field is Euclidean is closely related to questions concerning the class number and principality. Motzkin\cite{Motzkin1949} interpreted the Euclidean algorithm in terms of a set-theoretic filtration and proved that every Euclidean domain admits a minimal Euclidean function. He also showed that, for the ring of integers of a number field, the existence of an ordinal-valued Euclidean function is equivalent to the existence of a Euclidean function with values in the natural numbers.

Lenstra\cite{Lenstra1979} extended this framework from rings of integers to ideal classes and introduced the notion of a Euclidean ideal class. This notion formulates Euclidean division not only in the principal ideal class, but also in all ideal classes. From this point of view, it becomes possible to study Euclidity for elements of the class group of a Dedekind domain.

Treatman\cite{Treatman1998} further generalized this notion by formulating Euclidity for a set of ideal classes, rather than for a single ideal class, and introduced the notion of a Euclidean system. This notion is a natural multi-class generalization of Euclidean ideal classes and clarifies the relation between generating sets of the class group and the Euclidean algorithm.

The ideal class group is the ray class group of trivial modulus. It is therefore natural to extend Treatman's Euclidean systems to the setting of ray class groups. In this paper, we formulate Euclidean systems for ray classes and study their basic properties. This definition agrees with Treatman's definition when the modulus is trivial.

The first result of this paper is a ray class analogue of Treatman's theorem. The precise definitions will be given later, but the main point can be stated as follows.
\\

\begin{thm}\label{11} \n
Let \(K\) be a number field, \(\m\) a modulus of \(K\), and \(Cl_K^{\m}\) the ray class group modulo \(\m\). If a subset \(\mathcal{C}\subset Cl_K^{\m}\) is a Euclidean system of ray classes, then \(\mathcal{C}\) generates \(Cl_K^{\m}\).
\\
\end{thm}

There is also a deep existence theorem due to Lenstra\cite{Lenstra1977} concerning Euclidean ideal classes. Using Artin's generalized primitive root conjecture, which holds under GRH, Lenstra proved that a generator of the ideal class group of the ring of integers is a Euclidean ideal class.

Treatman extended this result to the setting of Euclidean systems. More precisely, he decomposed the ideal class group as
\[
Cl_K\simeq \bigoplus_{i=1}^r \mathbb{Z}/e_i\mathbb{Z}~~(e_i|e_{i+1})
\]
and proved, under the assumption that the rank of the unit group is at least \(\max\{1,r-1\}\), that the set of ideal classes corresponding to the standard generators \((1~\text{mod}~e_i)\) of the right-hand side is a Euclidean system. Here, the invariant factor condition \(e_i\mid e_{i+1}\) gives a standard presentation in which the number \(r\) of generators is minimal; it does not play an essential role in the construction of the Euclidean system itself.

The second main result of this paper gives a ray class analogue of Treatman's existence theorem. Its precise formulation uses the Motzkin filtration associated with a subset of ray classes, the partial minimal Euclidean function on it, and an admissibility condition. These notions will be introduced later. Roughly speaking, the admissibility condition ensures that sufficiently many prime ideals lie in the first layer of Motzkin's filtration. Under this condition, and assuming GRH, we prove that generating sets of ray class groups are Euclidean systems.

In particular, we obtain the following sufficient condition.
\\

\begin{thm}\label{12} \n
Assume GRH. Let \(K\) be a totally real Galois number field of degree \(n\ge 3\), and let \(p\) be an odd rational prime which does not split completely in \(K\). Then, for every \(N>0\), any generating set of \(Cl_K^{(p)^N}\) is a Euclidean system of ray classes.
\\
\end{thm}

The paper is organized as follows. In Section 2, we recall the notation and basic facts concerning Euclidean ideal classes, Euclidean systems, and ray class groups. In Section 3, we define Euclidean systems for ray classes and explain their compatibility with Treatman's definition. In Section 4, we introduce the Motzkin filtration, establish basic results on Euclidean systems, and finally prove that every Euclidean system generates the corresponding ray class group. In Section 5, we invoke Lenstra's theorem and prove an existence theorem for Euclidean systems under GRH.
\\

\section{Preliminaries}

Throughout this paper, \(\mathbb{N}\) denotes the set of non-negative integers, and \(\mathbf{ON}\) denotes the proper class of all ordinal numbers. Let \(R\) be a Dedekind domain with ideal class group \(Cl(R)\). Let \(K\) be a number field with ring of integers \(\ded_K\) and unit group \(E_K\). Let \(\mathfrak{m}\) be a finite modulus of \(K\). We denote by \(J_K^{\mathfrak{m}}\) the group of fractional ideals of \(K\) which are relatively prime to \(\mathfrak{m}\), by \(K^{\mathfrak{m}}\subset K^\ast\) the subgroup of all elements congruent to \(1\) modulo \(\mathfrak{m}\) and by \(P_K^{\mathfrak{m}}\) the subgroup of principal fractional ideals generated by elements congruent to \(1\) modulo \(\mathfrak{m}\). We write
\[
Cl_K^{\m}=J_K^{\m}/P_K^{\m}
\]
for the ray class group of \(K\) modulo \(\m\). We denote by \(\text{Int}(J_K^{\m})\) the set of integral ideals belonging to \(J_K^{\m}\). For a ray class \(C\in Cl_K^{\m}\), we write
\[
\text{Int}(C)=\{\mathfrak{a}\in C~|~\mathfrak{a}\text{ is integral}\}.
\]
We also write \(E^{\m}\) for the group of units congruent to \(1\) modulo \(\m\).

We first recall the definition of Euclidean ideal classes.
\\

\begin{defn}\label{21} \n
Let \(C=[\mathfrak{c}]\in Cl(R)\), where \(\mathfrak{c}\) is an integral ideal. A function
\[
\varphi:\{\text{non-zero integral ideals of }R\}\to \mathbb{N}
\]
is called a \textit{Euclidean function} for \(C\) if, for every integral ideal \((1)\neq \mathfrak{b}\) and every
\[
\gamma\in \mathfrak{b}^{-1}\mathfrak{c}-\mathfrak{c},
\]
there exists \(q\in \mathfrak{c}\) such that
\[
\varphi(\mathfrak{b}\mathfrak{c}^{-1}(\gamma-q))<\varphi(\mathfrak{b}).
\]
If such a function exists, then \(C\) is called a \textit{Euclidean ideal class}.
\\
\end{defn}

This definition is independent of the choice of the integral ideal \(\mathfrak{c}\in C\).

Treatman extended this notion to Euclidean systems.
\\

\begin{defn}\label{22} \n
Let \(\mathcal{C}=\{C_1,\dots,C_r\}\subset Cl(R)\) be a subset of ideal classes. Choose pairwise coprime integral ideals \(\mathfrak{c}_i\in C_i\), and put
\[
\mathfrak{c}=\prod_{i=1}^r \mathfrak{c}_i.
\]
A function
\[
\varphi:\{\text{non-zero integral ideals of }R\}\to \mathbb{N}
\]
is called a \textit{Euclidean function for \(\mathcal{C}\)} if, for every integral ideal \((1)\neq \mathfrak{b}\) and every
\[
\gamma\in \mathfrak{b}^{-1}\mathfrak{c}-\mathfrak{c},
\]
there exist \(i\in\{1,\dots,r\}\) and \(q_i\in \mathfrak{c}_i\) with \(q_i\neq \gamma\) such that
\[
\varphi(\mathfrak{b}\mathfrak{c}_i^{-1}(\gamma-q_i))<\varphi(\mathfrak{b}).
\]
If such a function exists, then \(\mathcal{C}\) is called a \textit{Euclidean system}.
\\
\end{defn}

This definition is also independent of the choice of the integral ideals \(\mathfrak{c}_i\in C_i\), but it requires the coprime condition.

The following result is Treatman's extension of Lenstra's theorem for Euclidean ideal classes.
\\

\begin{thm}\label{23} \n
Let \(\mathcal{C}=\{C_1,\dots,C_r\}\subset Cl_K\) be a Euclidean system and let \(\varphi\) be a Euclidean function for \(\mathcal{C}\). Then, for any \(B\in Cl_K\), there exist integers \(n_1,\dots,n_r\in\mathbb{N}\) such that
\[
0\le \sum_{i=1}^r n_i\le\min\{\varphi(\mathfrak{b})~|~\mathfrak{b}\in B,~\mathfrak{b}\text{ is integral}\}
\]
and
\[
B=\prod_{i=1}^r C_i^{n_i}.
\]
In particular, \(\mathcal{C}\) generates the ideal class group.
\\
\end{thm}

We also recall Graves' Motzkin construction for Euclidean ideal classes\cite{Graves2009}.
\\

\begin{defn}\label{24} \n
Let \(C=[\mathfrak{c}]\in Cl(R)\). Define
\[
A_{C,0}=\{(1)\}.
\]
For \(n>0\), define \(A_{C,n}\) inductively by
\[
A_{C,n}=A_{C,n-1}\cup\left\{\mathfrak{b}~\middle|~\begin{array}{l}\mathfrak{b}\text{ is a non-zero integral ideal of } R, \\ \text{for every } \gamma\in \mathfrak{b}^{-1}\mathfrak{c}-\mathfrak{c}, \\ \text{there exists } q\in \mathfrak{c} \text{ such that }\mathfrak{b}\mathfrak{c}^{-1}(\gamma-q)\in A_{C,n-1}\end{array}\right\}.
\]
Finally, set
\[
A_{C,\omega}=\bigcup_{n=0}^{\infty}A_{C,n}.
\]
The filtration \(\{A_{C,n}\}\), or its union \(A_{C,\omega}\), is called \textit{Motzkin's construction for \(C\)}.
\\
\end{defn}

Note that these \(A_{C,n}\) and \(A_{C,\omega}\) are the sets of all inverse ideals of the original Motzkin construction as defined by Graves.
\\

\section{Euclidean systems of ray classes}

We first record a correspondence between unit orbits of congruence sets and integral ideals in a ray class. This correspondence motivates the definition of \(S_{\mathfrak{m}}(\tilde{\mathfrak{c}})\) below.
\\

\begin{lem}\label{31} \n
Let \(0\neq\alpha\in\ded_K\), and let \(\mathfrak{c}\) and \(\mathfrak{n}\) be non-zero integral ideals such that
\[
(\alpha,\mathfrak{c}\mathfrak{n})=\mathfrak{c}.
\]
Then the map
\[
E_K\backslash(E_K\alpha+\mathfrak{c}\mathfrak{n})\longrightarrow \left\{\mathfrak{a}~\middle|~\mathfrak{a}\in \mathfrak{c}^{-1}(\alpha)P_K^{\mathfrak{n}},~\mathfrak{a}\text{ is integral}\right\}
\]
defined by
\[
E_K\delta\longmapsto \mathfrak{c}^{-1}(\delta)
\]
is a bijection.
\end{lem}

\begin{proof}
Let \(\delta\in E_K\alpha+\mathfrak{c}\mathfrak{n}\). Then \(\delta=\varepsilon\alpha+\eta\) for some \(\varepsilon\in E_K\) and \(\eta\in \mathfrak{c}\mathfrak{n}\). Since \((\alpha,\mathfrak{c}\mathfrak{n})=\mathfrak{c}\), the integral ideal \(\mathfrak{c}^{-1}(\alpha)\) is relatively prime to \(\mathfrak{n}\). Hence,
\[
\frac{\delta}{\varepsilon\alpha}=1+\frac{\eta}{\varepsilon\alpha}\equiv 1 \pmod{\mathfrak{n}}.
\]
Therefore
\[
\frac{(\delta)}{(\alpha)}\in P_K^{\mathfrak{n}},
\]
and so \(\mathfrak{c}^{-1}(\delta)\in \mathfrak{c}^{-1}(\alpha)P_K^{\mathfrak{n}}\). Moreover,
\[
E_K\alpha+\mathfrak{c}\mathfrak{n}\subset (\alpha,\mathfrak{c}\mathfrak{n})=\mathfrak{c},
\]
so \(\mathfrak{c}^{-1}(\delta)\) is integral. Thus the map is constructed and is clearly well-defined.

The map is injective. Indeed, if \(\mathfrak{c}^{-1}(\delta)=\mathfrak{c}^{-1}(\delta')\), then \((\delta)=(\delta')\). Hence \(\delta'=u\delta\) for some unit \(u\in E_K\), and therefore \(\delta\) and \(\delta'\) define the same element of \(E_K\backslash(E_K\alpha+\mathfrak{c}\mathfrak{n})\).

It remains to prove surjectivity. We first assume that
\[
(\mathfrak{c},\mathfrak{n})=1.
\]
Let \(\mathfrak{a}\in \mathfrak{c}^{-1}(\alpha)P_K^{\mathfrak{n}}\) be an integral ideal. Then, there exists \(\gamma\in K^\ast\) such that \(\gamma\equiv 1 \pmod{\mathfrak{n}}\) and \(\mathfrak{a}=\mathfrak{c}^{-1}(\alpha)(\gamma)\). Put \(\delta=\alpha\gamma\), then \((\delta)=\mathfrak{a}\mathfrak{c}\). We have
\[
\delta-\alpha=\alpha(\gamma-1)\in (\alpha)\mathfrak{n}.
\]
Also, \(\delta\in\mathfrak{c}\mathfrak{a}\subset\mathfrak{c}\) and \(\alpha\in\mathfrak{c}\), so \(\delta-\alpha\in\mathfrak{c}\). Since \((\mathfrak{c},\mathfrak{n})=1\) and \((\alpha,\mathfrak{c}\mathfrak{n})=\mathfrak{c}\), we have
\[
\delta-\alpha\in \mathfrak{c}\cap (\alpha)\mathfrak{n}=\mathfrak{c}\mathfrak{n}.
\]
Hence, \(\delta-\alpha\in\mathfrak{c}\mathfrak{n}\), and therefore \(\delta\in E_K\alpha+\mathfrak{c}\mathfrak{n}\). Moreover, \(\mathfrak{c}^{-1}(\delta)=\mathfrak{a}\). This proves surjectivity in the case \((\mathfrak{c},\mathfrak{n})=1\).

We now prove the general case. Let
\[
\mathfrak{d}=\prod_{\p\mid \mathfrak{n}}\mathfrak{p}^{v_{\mathfrak{p}}(\mathfrak{c})}
\]
be the part of \(\mathfrak{c}\) supported on the prime ideals dividing \(\mathfrak{n}\). Choose an integral ideal \(\mathfrak{e}\), prime to \(\mathfrak{c}\mathfrak{n}\), such that \(\mathfrak{d}^{-1}\mathfrak{e}\) is principal, say \((\lambda)=\mathfrak{d}^{-1}\mathfrak{e}\) for some \(\lambda\in K^\ast\). Put
\[
\tilde{\alpha}=\lambda\alpha,~\tilde{\mathfrak{c}}=\lambda\mathfrak{c}.
\]
Then \(\tilde{\mathfrak{c}}=\lambda\mathfrak{c}=\mathfrak{c}\mathfrak{d}^{-1}\mathfrak{e}\) is an integral ideal, and by construction
\[
(\tilde{\mathfrak{c}},\mathfrak{n})=1,~(\tilde{\alpha},\tilde{\mathfrak{c}}\mathfrak{n})=\tilde{\mathfrak{c}}.
\]
Multiplication by \(\lambda\) gives a bijection
\[
E_K\backslash(E_K\alpha+\mathfrak{c}\mathfrak{n})\longrightarrow E_K\backslash(E_K\tilde{\alpha}+\tilde{\mathfrak{c}}\mathfrak{n}),~E_K\delta\longmapsto E_K\lambda\delta.
\]
On the other hand, \(\tilde{\mathfrak{c}}^{-1}(\tilde{\alpha})=(\lambda\mathfrak{c})^{-1}(\lambda\alpha)=\mathfrak{c}^{-1}(\alpha)\), and for every \(\delta\) we have \(\tilde{\mathfrak{c}}^{-1}(\lambda\delta)=\mathfrak{c}^{-1}(\delta)\). Therefore, the diagram
\[
\xymatrix{
E_K\backslash(E_K\alpha+\mathfrak{c}\mathfrak{n}) \ar[rd] \ar[d]_{\lambda} & \\
E_K\backslash(E_K\tilde{\alpha}+\tilde{\mathfrak{c}}\mathfrak{n}) \ar[r] \ar@{}[ru]|(.34){\circlearrowright} &  \left\{\mathfrak{a}~\middle|~\mathfrak{a}\in \mathfrak{c}^{-1}(\alpha)P_K^{\mathfrak{n}},~\mathfrak{a} \text{ is integral}\right\}
}
\]
commutes. Since the lower horizontal map is bijective by the coprime case, the upper horizontal map is also bijective.
\\
\end{proof}

We now introduce the notation used in the definition of Euclidean systems for ray classes. Let \(\widetilde{\mathfrak{c}}\) be a fractional ideal prime to \(\mathfrak{m}\). Write
\[
\tilde{\mathfrak{c}}=\frac{\mathfrak{a}}{\mathfrak{d}},
\]
where \(\mathfrak{a}\) and \(\mathfrak{d}\) are coprime integral ideals. We call \(\mathfrak{a}\) and \(\mathfrak{d}\) the numerator and the denominator of \(\widetilde{\mathfrak{c}}\), respectively. Define
\[
S_{\mathfrak{m}}(\tilde{\mathfrak{c}})=\mathfrak{d}^{-1}\text{Int}(\mathfrak{a}P_K^{\mathfrak{d}\mathfrak{m}}).
\]
Let \(\mathcal{C}=\{C_1,\dots,C_r\}\subset Cl_K^{\mathfrak{m}}\) be a subset, and let \((1)\neq \mathfrak{b}\in \text{Int}(J_K^{\mathfrak{m}})\). Define \(D(\mathcal{C},\mathfrak{b})\) to be the set of all \(r\)-tuples
\[
\left(\prod_{j\neq i}\mathfrak{c}_j\right)_{i=1}^r
\]
where \(\mathfrak{c}_i\in \text{Int}(C_i)\) and the ideals \(\mathfrak{b},\mathfrak{c}_1,\dots,\mathfrak{c}_r\) are pairwise coprime. We also define
\[
I(\mathcal{C},\mathfrak{b})=D(\mathcal{C},\mathfrak{b})\cdot\left\{\tilde{\mathfrak{d}}~\middle|~\begin{array}{l}\tilde{\mathfrak{d}}\in (C_1\cdots C_r)^{-1}-\text{Int}\bigl((C_1\cdots C_r)^{-1}\bigr), \\ \tilde{\mathfrak{d}}\mathfrak{b}\text{ is integral}\end{array}\right\}.
\]
Equivalently, \(I(\mathcal{C},\mathfrak{b})\) is the set of all \(r\)-tuples of the form
\[
\left(\tilde{\mathfrak{d}}\prod_{j\neq i}\mathfrak{c}_j\right)_{i=1}^r,
\]
where \(\mathfrak{c}_i\in\text{Int}(C_i)\), the ideals \(\mathfrak{b},\mathfrak{c}_1,\dots,\mathfrak{c}_r\) are pairwise coprime, and \(\tilde{\mathfrak{d}}\in (C_1\cdots C_r)^{-1}-\text{Int}\bigl((C_1\cdots C_r)^{-1}\bigr)\) such that \(\tilde{\mathfrak{d}}\mathfrak{b}\) is integral.

When \(r=1\), we write
\[
I(C,\mathfrak{b})=\left\{\tilde{\mathfrak{d}}~\middle|~\tilde{\mathfrak{d}}\in C^{-1}-\text{Int}(C^{-1}),~\tilde{\mathfrak{d}}\mathfrak{b}\text{ is integral}\right\}
\]
and clearly
\[
I(\mathcal{C},\mathfrak{b})=D(\mathcal{C},\mathfrak{b})\cdot I(C_1\dots C_r,\mathfrak{b}).
\]
\\

\begin{defn}\label{32} \n
Let \(\mathcal{C}=\{C_1,\dots,C_r\}\subset Cl_K^{\m}\) be a subset. A function
\[
\varphi:\text{Int}(J_K^{\mathfrak m})\longrightarrow \mathbf{ON}
\]
is called \textit{a Euclidean function for \(\mathcal{C}\)} if, for any \((1)\neq\mathfrak{b}\in \text{Int}(J_K^{\mathfrak m})\) and any \((\tilde{\mathfrak{c}}_1,\dots,\tilde{\mathfrak{c}}_r)\in I(\mathcal{C},\mathfrak{b})\), there exist \(i\in\{1,\dots,r\}\) and
\[
\mathfrak{x}_i\in S_{\mathfrak{m}}(\tilde{\mathfrak{c}}_i)
\]
such that
\[
\varphi(\mathfrak{x}_i\mathfrak{b})<\varphi(\mathfrak{b}).
\]
If such a function exists, then \(\mathcal{C}\) is called \textit{a Euclidean system of ray classes}.
\\
\end{defn}

The following proposition shows that, in the case of trivial modulus, the above definition recovers Treatman's Euclidean systems. Here we restrict to \(\mathbb{N}\)-valued Euclidean functions; later, in Section 4, we will show that for rings of integers of number fields, the existence of an ordinal-valued Euclidean function is equivalent to the existence of an \(\mathbb{N}\)-valued one.
\\

\begin{prop}\label{33} \n
Suppose that \(\mathfrak{m}=(1)\). Let \(\mathcal{C}=\{C_1,\dots,C_r\}\subset Cl_K\) and let
\[
\varphi:\{\text{non-zero integral ideals of }\ded_K\}\to \mathbb{N}
\]
be a function. Then \(\varphi\) is a Euclidean function for \(\mathcal{C}\) in the above sense if and only if it is a Euclidean function for \(\mathcal{C}\) in the sense of Treatman.
\end{prop}

\begin{proof} \n
We compare the two Euclidean conditions. Choose pairwise coprime integral ideals \(\mathfrak{c}_i\in C_i\) and put
\[
\mathfrak{c}=\prod_{i=1}^r \mathfrak{c}_i.
\]
Treatman's condition says that, for every non-zero integral ideal \((1)\neq \mathfrak{b}\) and every \(\gamma\in \mathfrak{b}^{-1}\mathfrak{c}-\mathfrak{c}\), there exist \(i\in\{1,\dots,r\}\) and \(q_i\in \mathfrak{c}_i\) such that
\[
\varphi(\mathfrak{b}\mathfrak{c}_i^{-1}(\gamma-q_i))<\varphi(\mathfrak{b}).
\]
Since Treatman's definition is independent of the chosen representatives \(\mathfrak{c}_i\), we may allow the \(\mathfrak{c}_i\)'s to vary among pairwise coprime integral representatives of the classes \(C_i\), and also require them to be coprime to \(\mathfrak{b}\).

For fixed \(\gamma\), put 
\[
\tilde{\mathfrak{c}}_i=\frac{(\gamma)}{\mathfrak{c}_i}.
\]
Writing \(\ds \gamma=\frac{\alpha}{\beta}\) with \(0\neq\alpha,\beta\in\ded_K\), the preceding lemma identifies the unit orbits in \(E_K\gamma+\mathfrak{c}_i\) with the set \(S_{(1)}(\tilde{\mathfrak{c}}_i)\), more explicitly, the ideals of the form \(\mathfrak{c}_i^{-1}(\gamma-q_i),~q_i\in \mathfrak{c}_i\), are precisely the elements of \(S_{(1)}(\tilde{\mathfrak{c}}_i)\). Moreover, as the representatives \(\mathfrak{c}_i\) and the element \(\gamma\) vary, the resulting tuples \((\tilde{\mathfrak{c}}_1,\dots,\tilde{\mathfrak{c}}_r)\) are exactly the elements of \(I(\mathcal{C},\mathfrak{b})\). Indeed, if \(\ds \tilde{\mathfrak{d}}=\frac{(\gamma)}{\mathfrak{c}}\), then
\[
\tilde{\mathfrak c}_i=\frac{(\gamma)}{\mathfrak{c}_i}=\tilde{\mathfrak{d}}\prod_{j\neq i}\mathfrak{c}_j.
\]
The conditions
\[
\gamma\in \mathfrak{b}^{-1}\mathfrak{c}-\mathfrak{c}
\]
are equivalent to
\[
\tilde{\mathfrak{d}}\mathfrak{b}\text{ is integral and}~\tilde{\mathfrak{d}}\notin\text{Int}\bigl((C_1\cdots C_r)^{-1}\bigr).
\]
Thus the possible tuples \(\ds \left(\frac{(\gamma)}{\mathfrak{c}_i}\right)_{i=1}^r\) are precisely the elements of \(I(\mathcal{C},\mathfrak{b})\).

Therefore Treatman's Euclidean condition is equivalent to the condition in the above definition with \(\mathfrak{m}=(1)\). This proves the assertion.
\\
\end{proof}

\section{Motzkin filtration and basic properties}

Throughout this section, unless otherwise stated, let
\[
\mathcal C=\{C_1,\dots,C_r\}\subset Cl_K^{\mathfrak m}
\]
be a subset of the ray class group.

Euclidean functions are semi-multiplicative.
\\

\begin{prop}\label{41} \n
Let \(\mathcal{C}\) be a Euclidean system, and let \(\varphi\) be a Euclidean function for \(\mathcal{C}\). Then, for any \(\mathfrak{b}_1,\mathfrak{b}_2\in \text{Int}(J_K^{\mathfrak{m}})\), we have
\[
\varphi(\mathfrak{b}_1)\le \varphi(\mathfrak{b}_1\mathfrak{b}_2),
\]
and equality holds if and only if \(\mathfrak{b}_2=(1)\). In particular,
\[
\varphi^{-1}(\min \text{Im}(\varphi))={(1)}.
\]
\end{prop}

\begin{proof}
When \(\mathfrak{b}_1=(1)\), the inequality is clear. Assume that \(\mathfrak{b}_1\neq (1)\). Choose an integral ideal \(\mathfrak{b}_2^0\) such that
\[
\varphi(\mathfrak{b}_1\mathfrak{b}_2^0)=\min\{\varphi(\mathfrak{b}_1\mathfrak{b}_2)~|~\mathfrak{b}_2\in \text{Int}(J_K^{\mathfrak{m}})\}.
\]
Suppose that \(\mathfrak{b}_2^0\neq (1)\). Then we can take \(\tilde{\mathfrak{d}}\in I(C_1\cdots C_r,\mathfrak{b}_2^0)\), and we can also take pairwise coprime ideals \(\mathfrak{c}_i\in \text{Int}(C_i)\). Then \(\tilde{\mathfrak{d}}\in I(C_1\cdots C_r,\mathfrak{b}_1\mathfrak{b}_2^0)\). Hence there exist \(i=1,\dots,r\) and
\[
\mathfrak{x}_i\in S_{\mathfrak{m}}\left(\tilde{\mathfrak{d}}\prod_{j\neq i}\mathfrak{c}_j\right)
\]
such that \(\varphi(\mathfrak{x}_i\mathfrak{b}_1\mathfrak{b}_2^0)<\varphi(\mathfrak{b}_1\mathfrak{b}_2^0)\). Since \(\mathfrak{x}_i\mathfrak{b}_2^0\in \text{Int}(J_K^{\mathfrak{m}})\), this contradicts the minimality of \(\varphi(\mathfrak{b}_1\mathfrak{b}_2^0)\). Therefore \(\mathfrak{b}_2^0=(1)\), and the inequality follows.
\\
\end{proof}

\begin{prop}\label{42} \n
For a Euclidean system \(\mathcal{C}\), the function
\[
\varphi_{\mathcal{C},\min}:\text{Int}(J_K^{\mathfrak{m}})\to \mathbf{ON};~\mathfrak{b}\mapsto \min\{\varphi(\mathfrak{b})~|~\varphi\text{ is a Euclidean function for }\mathcal{C}\}
\]
is a Euclidean function for \(\mathcal{C}\).
\end{prop}

\begin{proof}
For any \((1)\neq \mathfrak{b}\in \text{Int}(J_K^{\mathfrak{m}})\), choose a Euclidean function \(\varphi_{\mathfrak{b}}\) for \(\mathcal{C}\) such that \(\varphi_{\mathcal{C},\min}(\mathfrak{b})=\varphi_{\mathfrak{b}}(\mathfrak{b})\). Then, for any \((\tilde{\mathfrak{c}}_1,\dots,\tilde{\mathfrak{c}}_r)\in I(\mathcal{C},\mathfrak{b})\), there exist \(i=1,\dots,r\) and \(\mathfrak{x}_i\in S_{\mathfrak{m}}(\tilde{\mathfrak{c}}_i)\) such that \(\varphi_{\mathfrak{b}}(\mathfrak{x}_i\mathfrak{b})<\varphi_{\mathfrak{b}}(\mathfrak{b})\). Therefore
\[
\varphi_{\mathcal{C},\min}(\mathfrak{x}_i\mathfrak{b})\le \varphi_{\mathfrak{b}}(\mathfrak{x}_i\mathfrak{b})<\varphi_{\mathfrak{b}}(\mathfrak{b})=\varphi_{\mathcal{C},\min}(\mathfrak{b}).
\]
Thus \(\varphi_{\mathcal{C},\min}\) is a Euclidean function for \(\mathcal{C}\).
\\
\end{proof}

Euclidean systems of ray classes have the minimal Euclidean function.
\\

\begin{prop}\label{43} \n
Let \(\mathcal{C}\) be a Euclidean system, and let \(\varphi\) be a Euclidean function for \(\mathcal{C}\). Then an integral ideal \(\mathfrak{b}\), other than \((1)\), whose value with respect to \(\varphi\) is minimal belongs to one of \(C_1,\dots,C_r\). In particular, for any subset \(\mathcal{C}'\subset Cl_K^{\mathfrak{m}}\) such that \(\mathcal{C}'\cap\mathcal{C}=\emptyset\), the function \(\varphi\) is not a Euclidean function for \(\mathcal{C}'\).
\end{prop}

\begin{proof}
Take \((\tilde{\mathfrak{c}}_1,\dots,\tilde{\mathfrak{c}}_r)\in I(\mathcal{C},\mathfrak{b})\). Then there exist \(i=1,\dots,r\) and \(\mathfrak{x}_i\in S_{\mathfrak{m}}(\tilde{\mathfrak{c}}_i)\) such that \(\varphi(\mathfrak{x}_i\mathfrak{b})<\varphi(\mathfrak{b})\). By the minimality of \(\varphi(\mathfrak{b})\), we have \(\mathfrak{x}_i\mathfrak{b}=(1)\). Hence \(\mathfrak{b}\in [\mathfrak{x}_i^{-1}]=C_i\).
\\
\end{proof}

We define Motzkin hierarchy for systems of ray classes.
\\

\begin{defn}\label{44} \n
Define
\begin{align*}
A_{\mathcal{C},0}&=\{(1)\}, \\
A_{\mathcal{C},\mu}&=\begin{cases} A_{\mathcal{C},\mu-1}\cup \left\{\mathfrak{b}\in \text{Int}(J_K^{\mathfrak{m}})~\middle|~\begin{array}{l}
\forall (\tilde{\mathfrak{c}}_1,\dots,\tilde{\mathfrak{c}}_r)\in I(\mathcal{C},\mathfrak{b}), \\
\exists i=1,\dots,r,~\exists \mathfrak{x}_i\in S_{\mathfrak{m}}(\tilde{\mathfrak{c}}_i) \\
\text{such that }\mathfrak{x}_i\mathfrak{b}\in A_{\mathcal{C},\mu-1}\end{array}\right\} & (\mu\text{ is a successor ordinal}) \\ 
\ds \bigcup_{\nu<\mu}A_{\mathcal{C},\nu} & (\mu\text{ is a limit ordinal})
\end{cases}
\end{align*}
We call this the \textit{Motzkin hierarchy} of \(\mathcal{C}\). We write
\[
A_{\mathcal{C},\textbf{ON}}=\{\mathfrak{b}\in \text{Int}(J_K^{\mathfrak{m}})~|~\exists \mu\in\textbf{ON},~\mathfrak{b}\in A_{\mathcal{C},\mu}\},
\]
and
\[
\varphi_{\mathcal{C}}:A_{\mathcal{C},\mathbf{ON}}\to \mathbf{ON};~\mathfrak{b}\mapsto \min\{\mu~|~\mathfrak{b}\in A_{\mathcal{C},\mu}\}.
\]
\\
\end{defn}

\begin{prop}\label{45} \n
(i) For a ray class \(C\in Cl_K^{\mathfrak{m}}\),
\[
A_{C,1}-A_{C,0}=\left\{\mathfrak{p}\in \mathfrak{Spec}(\ded_K)-\{(0)\}~\middle|~\mathfrak{p}\in C,~E^{\mathfrak{m}}\to \kappa_{\mathfrak{p}}^\ast\text{ is surjective}\right\}.
\]
(ii) For a subset \(\mathcal{C}=\{C_1,\dots,C_r\}\subset Cl_K^{\mathfrak{m}}\) of ray classes,
\[
A_{\mathcal{C},1}=\bigcup_{i=1}^r A_{C_i,1}.
\]
\end{prop}

\begin{proof}
(i) Suppose that \((1)\neq \mathfrak{b}\in A_{C,1}\) is not a prime ideal. Choose \(\mathfrak{b}\subsetneq \mathfrak{b}'\subsetneq (1)\), and choose \(\tilde{\mathfrak{a}}\in \text{Int}([\mathfrak{b}']C^{-1})\) which is relatively prime to \(\mathfrak{b}'\). Then \(\ds \frac{\tilde{\mathfrak{a}}}{\mathfrak{b}'}\in I(C,\mathfrak{b})\). Hence there exists \(\ds \mathfrak{x}\in S_{\mathfrak{m}}\left(\frac{\tilde{\mathfrak{a}}}{\mathfrak{b}'}\right)\) such that \(\mathfrak{x}\mathfrak{b}\in A_{C,0}={(1)}\). Thus
\[
\mathfrak{b}^{-1}=\mathfrak{x}\in S_{\mathfrak{m}}\left(\frac{\tilde{\mathfrak{a}}}{\mathfrak{b}'}\right)={\mathfrak{b}'}^{-1}
\text{Int}(\tilde{\mathfrak{a}}P_K^{\mathfrak{b}'\mathfrak{m}}).
\]
This is a contradiction, since the denominator of \(\mathfrak{b}^{-1}\) is strictly larger than the denominator \({\mathfrak{b}'}^{-1}\). Therefore
\[
A_{C,1}-A_{C,0}\subset\mathfrak{Spec}(\ded_K)-{(0)}.
\]

For any prime ideal \(\mathfrak{p}\nmid\mathfrak{m}\), the definition of \(A_{C,1}\) gives
\[
\mathfrak p\in A_{C,1}~\Longleftrightarrow~\forall \tilde{\mathfrak{c}}\in I(C,\mathfrak{p}),~\mathfrak{p}^{-1}\in S_{\mathfrak{m}}(\tilde{\mathfrak{c}}).
\]
Writing \(\ds \tilde{\mathfrak{c}}=\frac{\tilde{\mathfrak{a}}}{\mathfrak{p}}\), this is equivalent to
\[
\forall \tilde{\mathfrak{a}}\in \text{Int}(\mathfrak{p} C^{-1})\text{ with } \mathfrak{p}\nmid \tilde{\mathfrak{a}},~(1)\in \text{Int}(\tilde{\mathfrak{a}}P_K^{\mathfrak{p}\mathfrak{m}}).
\]
Equivalently,
\[
\text{Int}(\mathfrak{p} C^{-1})\cap J_K^{\mathfrak{p}\mathfrak{m}}\subset P_K^{\mathfrak{p}\mathfrak{m}}.
\]
This condition holds if and only if
\[
\mathfrak{p}\in C~\text{and}~P_K^{\mathfrak{m}}\cap J_K^{\mathfrak{p}\mathfrak{m}}\subset P_K^{\mathfrak{p}\mathfrak{m}}.
\]
The latter inclusion says precisely that every principal ideal generated by an element congruent to \(1\) modulo \(\mathfrak{m}\) and prime to \(\mathfrak{p}\mathfrak{m}\) has a generator congruent to \(1\) modulo \(\mathfrak{p}\mathfrak{m}\). In other words, if
\[
K^{\mathfrak{m}}_{(\mathfrak{p})}:=\{\alpha\in K^{\mathfrak{m}}~|~v_{\mathfrak{p}}(\alpha)=0\},
\]
then for every \(\alpha\in K^{\mathfrak{m}}_{(\mathfrak{p})}\), there exist \(\beta\in K^{\mathfrak{p}\mathfrak{m}}\) and \(\varepsilon\in E^{\mathfrak{m}}\) such that \(\alpha=\varepsilon\beta\). Equivalently, the natural map
\[
E^{\mathfrak{m}}\longrightarrow K^{\mathfrak{m}}_{(\mathfrak{p})}/K^{\mathfrak{p}\mathfrak{m}}
\]
is surjective. Since the reduction map induces an isomorphism \(K^{\mathfrak{m}}_{(\mathfrak{p})}/K^{\mathfrak{p}\mathfrak{m}}\cong\kappa_{\mathfrak{p}}^{\ast}\), this is equivalent to the surjectivity of \(E^{\mathfrak{m}}\to \kappa_{\mathfrak{p}}^{\ast}\). We obtain
\[
\mathfrak{p}\in A_{C,1}~\Longleftrightarrow~\mathfrak{p}\in C\text{ and }E^{\mathfrak{m}}\to \kappa_{\mathfrak{p}}^{\ast}\text{ is surjective}.
\]
Therefore,
\[
A_{C,1}-A_{C,0}=\left\{\mathfrak{p}\in \mathfrak{Spec}(\ded_K)-\{(0)\}~\middle|~\mathfrak{p}\in C,~E^{\mathfrak{m}}\to \kappa_{\mathfrak{p}}^\ast\text{ is surjective}\right\}.
\]

(ii) We have
\begin{align*}
A_{\mathcal{C},1}&=\left\{\mathfrak{b}\in \text{Int}(J_K^{\mathfrak{m}})~\middle|~
\begin{array}{l}
\forall (\tilde{\mathfrak{c}}_1,\dots,\tilde{\mathfrak{c}}_r)\in I(\mathcal{C},\mathfrak{b}), \\
\exists i=1,\dots,r,~\exists \mathfrak{x}_i\in S_{\mathfrak{m}}(\tilde{\mathfrak{c}}_i) \\
\text{ such that }\mathfrak{x}_i\mathfrak{b}=(1)
\end{array}
\right\} \\
&=\left\{\mathfrak{b}\in \text{Int}(J_K^{\mathfrak{m}})~\Bigg|~\forall (\tilde{\mathfrak{c}}_1,\dots,\tilde{\mathfrak{c}}_r)\in I(\mathcal{C},\mathfrak{b}),~\mathfrak{b}^{-1}\in \bigcup_{i=1}^r S_{\mathfrak{m}}(\tilde{\mathfrak{c}}_i)\right\}.
\end{align*}
For any \(i=1,\dots,r\) and \(\mathfrak{p}\in A_{C_i,1}\), and for any \((\tilde{\mathfrak{c}}_1,\dots,\tilde{\mathfrak{c}}_r)\in I(\mathcal{C},\mathfrak{p})\), we have \(\tilde{\mathfrak{c}}_i\in I(C_i,\mathfrak{p})\). Hence \(\mathfrak{p}^{-1}\in S_{\mathfrak{m}}(\tilde{\mathfrak{c}}_i)\), and so \(\mathfrak{p}\in A_{\mathcal{C},1}\). Thus 
\[
\bigcup_{i=1}^r A_{C_i,1}\subset A_{\mathcal{C},1}.
\]

Conversely, let \((1)\neq \mathfrak{b}\in A_{\mathcal{C},1}\). Choose \((\tilde{\mathfrak{c}}_1,\dots,\tilde{\mathfrak{c}}_r)\in I(\mathcal{C},\mathfrak{b})\). Then there exists \(i=1,\dots,r\) such that \(\mathfrak{b}^{-1}\in S_{\mathfrak{m}}(\tilde{\mathfrak{c}}_i)\). Then, \(\mathfrak{b}^{-1}\in S_{\mathfrak{m}}(\tilde{\mathfrak{c}}_i)\subset C_i^{-1}\), and hence \(\mathfrak{b}\in C_i\). Thus \(\ds \mathfrak{b}\notin \bigcup_{j\neq i}C_j\), and for any \((\tilde{\mathfrak{c}}_1,\dots,\tilde{\mathfrak{c}}_r)\in I(\mathcal{C},\mathfrak{b})\), we have \(\mathfrak{b}^{-1}\in S_{\mathfrak{m}}(\tilde{\mathfrak{c}}_i)\).

Suppose that \(\mathfrak{b}\) is not a prime ideal. Choose ideals \(\mathfrak{c}_j\in \text{Int}(C_j)\) such that \(\mathfrak{b},\mathfrak{c}_1,\dots,\mathfrak{c}_r\) are pairwise coprime. Choose \(\mathfrak{b}\subsetneq \mathfrak{b}'\subsetneq (1)\), and choose \(\tilde{\mathfrak{a}}\in\text{Int}([\mathfrak{b}'](C_1\cdots C_r)^{-1})\) which is relatively prime to \(\mathfrak{b}'\). Put
\[
(\tilde{\mathfrak{c}}_1,\dots,\tilde{\mathfrak{c}}_r)=\left(\frac{\tilde{\mathfrak{a}}}{\mathfrak{b}'}\prod_{k\neq i}\mathfrak{c}_k
\right)_i\in I(\mathcal{C},\mathfrak{b}).
\]
Then
\[
\mathfrak{b}^{-1}\in S_{\mathfrak{m}}\left(\frac{\tilde{\mathfrak{a}}}{\mathfrak{b}'}\prod_{k\neq i}\mathfrak{c}_k\right)={\mathfrak{b}'}^{-1}
\text{Int}\left(\tilde{\mathfrak{a}}\prod_{k\neq i}\mathfrak{c}_kP_K^{\mathfrak{b}'\mathfrak{m}}\right).
\]
This is a contradiction, since the denominator of \(\mathfrak{b}^{-1}\) is strictly larger than the denominator \({\mathfrak{b}'}^{-1}\). Therefore \(\mathfrak{b}\) is a prime ideal.

Furthermore, for any \(\overline{\alpha}\in \kappa_{\mathfrak{b}}^\ast\), by the Chinese remainder theorem, there exists \(\ds \alpha\in
\left(\prod_{k\neq i}\mathfrak{c}_k-\mathfrak{b}\right)\cap K^{\mathfrak{m}}\) such that \((\alpha \bmod \mathfrak{b})=\overline{\alpha}\). Put
\[
\tilde{\mathfrak{a}}=\alpha\left(\prod_{k\neq i}\mathfrak{c}_k\right)^{-1}\in \text{Int}(C_i(C_1\cdots C_r)^{-1}).
\]
Since
\[
\mathfrak{b}^{-1}\in S_{\mathfrak{m}}\left(\frac{\tilde{\mathfrak{a}}}{\mathfrak{b}}\prod_{k\neq i}\mathfrak{c}_k\right)=S_{\mathfrak{m}}\left(\frac{(\alpha)}{\mathfrak{b}}\right)=\mathfrak{b}^{-1}\text{Int}((\alpha)P_K^{\mathfrak{b}\mathfrak{m}}),
\]
we have \((1)\in \text{Int}((\alpha)P_K^{\mathfrak{b}\mathfrak{m}})\). Therefore \((\alpha)\in P_K^{\mathfrak{b}\mathfrak{m}}\). Hence there exist \(\beta\in K^{\mathfrak{b}\mathfrak{m}}~\text{and}~\varepsilon\in E_K\) such that \(\alpha=\varepsilon\beta\). Then \(\ds \varepsilon=\frac{\alpha}{\beta}\in E^{\mathfrak{m}}\), and \(\varepsilon\equiv \alpha \pmod{\mathfrak{b}}\). Thus \(E^{\mathfrak{m}}\to \kappa_{\mathfrak{b}}^\ast\) is surjective. Therefore \(\mathfrak{b}\in A_{C_i,1}\). Hence
\[
A_{\mathcal{C},1}\subset \bigcup_{i=1}^r A_{C_i,1}.
\]
\\
\end{proof}

\begin{prop}\label{46} \n
The Motzkin hierarchy of \(\mathcal{C}\) becomes stationary at \(A_{\mathcal{C},\omega}\).
\end{prop}

\begin{proof}
Choose pairwise coprime ideals \(\mathfrak{c}_i\in C_i\), and put \(\ds \mathfrak{c}=\prod_{i=1}^r\mathfrak{c}_i\). Then, for any \(\mathfrak{b}\in A_{\mathcal{C},\omega+1}\), the map 
\[
E_K\backslash (\mathfrak{b}^{-1}\mathfrak{c}-\mathfrak{c})/\mathfrak{c}\twoheadrightarrow\left\{(S_{\mathfrak{m}}(\tilde{\mathfrak{c}}_i))_i~\middle|~(\tilde{\mathfrak{c}}_1,\dots,\tilde{\mathfrak{c}}_r)\in I(\mathcal{C},\mathfrak{b})\right\},~E_K\gamma+\mathfrak{c}\mapsto \left(S_{\mathfrak{m}}(\gamma\mathfrak{c}_i^{-1})\right)_i
\]
is surjective. Hence, by the assumption, the set
\[
\left\{(S_{\mathfrak{m}}(\tilde{\mathfrak{c}}_i))_i~\middle|~(\tilde{\mathfrak{c}}_1,\dots,\tilde{\mathfrak{c}}_r)\in I(\mathcal{C},\mathfrak{b})\right\}
\]
is finite. Write this set \(\{(S_i^{(1)})_i,\dots,(S_i^{(s)})_i\}\). Then, for each \(j=1,\dots,s\), there exist \(i_j=1,\dots,r\) and \(\mathfrak{x}_{j,i_j}\in S_{i_j}^{(j)}\) such that \(\mathfrak{x}_{j,i_j}\mathfrak{b}\in A_{\mathcal{C},\omega}\). Therefore there exist \(n_{j,i}\in\mathbb{N}\) such that \(\mathfrak{x}_{j,i_j}\mathfrak{b}\in A_{\mathcal{C},n_{j,i}}\). Hence \(\mathfrak{b}\in A_{\mathcal{C},\max_{j,i}{n_{j,i}}+1}\subset A_{\mathcal{C},\omega}\). Thus
\[
A_{\mathcal{C},\omega+1}=A_{\mathcal{C},\omega}.
\]
\\
\end{proof}

\begin{prop}\label{47} \n
\[
\mathcal{C}\text{ is Euclidean}~\Leftrightarrow~\text{Int}(J_K^{\mathfrak{m}})=A_{\mathcal{C},\omega}~\Leftrightarrow~\mathcal{C}\text{ is }\mathbb{N}\text{-valued Euclidean}.
\]
In this case, \(\varphi_{\mathcal{C}}\) is the minimal Euclidean function for \(\mathcal{C}\).
\end{prop}

\begin{proof}
Suppose that \(\mathcal{C}\) is Euclidean, and choose a Euclidean function \(\varphi\) for \(\mathcal{C}\). When \(\mu=0\), at most \((1)\) can satisfy \(\varphi(\mathfrak{b})\le \mu\). Hence 
\[
\{\mathfrak{b}\in\text{Int}(J_K^{\mathfrak{m}})~|~\varphi(\mathfrak{b})\le \mu\}\subset \{(1)\}\subset A_{\mathcal{C},\mu}.\tag{$\ast$}
\]
For a natural number $\mu$, assume that the assertion $(\ast)$ holds for all \(\nu<\mu\). Let \((1)\neq \mathfrak{b}\in \text{Int}(J_K^{\mathfrak{m}})\) satisfy \(\varphi(\mathfrak{b})\le \mu\). If \(\varphi(\mathfrak{b})\le \mu-1\), then \(\mathfrak{b}\in A_{\mathcal{C},\mu-1}\subset A_{\mathcal{C},\mu}\). If \(\varphi(\mathfrak{b})=\mu\), then, for any \((\tilde{\mathfrak{c}}_1,\dots,\tilde{\mathfrak{c}}_r)\in I(\mathcal{C},\mathfrak{b})\), there exist \(i=1,\dots,r\) and \(\mathfrak{x}_i\in S_{\mathfrak{m}}(\tilde{\mathfrak{c}}_i)\) such that \(\varphi(\mathfrak{x}_i\mathfrak{b})<\varphi(\mathfrak{b})=\mu\). Hence \(\mathfrak{x}_i\mathfrak{b}\in A_{\mathcal{C},\mu-1}\), and therefore \(\mathfrak{b}\in A_{\mathcal{C},\mu}\). By induction, the assertion ($\ast$) follows for any natural number $\mu$. Consequently, the assertion
\[
\{\mathfrak{b}\in\text{Int}(J_K^{\mathfrak{m}})~|~\varphi(\mathfrak{b})<\mu\}\subset A_{\mathcal{C},\mu}.\tag{$\ast\ast$}
\]
holds, too. Next for an ordinal number $\mu$, assume that the assertion $(\ast\ast)$ holds for all \(\nu<\mu\). If $\mu$ is a limit ordinal,
\[
\{\mathfrak{b}\in\text{Int}(J_K^{\mathfrak{m}})~|~\varphi(\mathfrak{b})<\mu\}=\bigcup_{\nu<\mu}\{\mathfrak{b}\in\text{Int}(J_K^{\mathfrak{m}})~|~\varphi(\mathfrak{b})<\nu\}\subset \bigcup_{\nu<\mu}A_{\mathcal{C},\nu}=A_{\mathcal{C},\mu}.
\]
If $\mu$ is a successor ordinal, let \(\mathfrak{b}\in \text{Int}(J_K^{\mathfrak{m}})\) satisfy \(\varphi(\mathfrak{b})\le \mu-1\). If \(\varphi(\mathfrak{b})<\mu-1\), then \(\mathfrak{b}\in A_{\mathcal{C},\mu-1}\subset A_{\mathcal{C},\mu}\). If \(\varphi(\mathfrak{b})=\mu-1\), then, for any \((\tilde{\mathfrak{c}}_1,\dots,\tilde{\mathfrak{c}}_r)\in I(\mathcal{C},\mathfrak{b})\), there exist \(i=1,\dots,r\) and \(\mathfrak{x}_i\in S_{\mathfrak{m}}(\tilde{\mathfrak{c}}_i)\) such that \(\varphi(\mathfrak{x}_i\mathfrak{b})<\varphi(\mathfrak{b})=\mu-1\). Hence \(\mathfrak{x}_i\mathfrak{b}\in A_{\mathcal{C},\mu-1}\), and therefore \(\mathfrak{b}\in A_{\mathcal{C},\mu}\). Thus, in either case, ($\ast\ast$) holds on $\mu$. By transfinite induction, the assertion ($\ast\ast$) follows for all $\mu$. If \(\mu_0>\sup \varphi(\text{Int}(J_K^{\mathfrak{m}}))\), then
\[
\text{Int}(J_K^{\mathfrak{m}})=A_{\mathcal{C},\mu_0}=A_{\mathcal{C},\omega}.
\]

Conversely, suppose that \(\text{Int}(J_K^{\mathfrak{m}})=A_{\mathcal{C},\omega}\). Then, for any \((1)\neq \mathfrak{b}\in \text{Int}(J_K^{\mathfrak{m}})\), by the definition of the Motzkin hierarchy, \(\varphi_{\mathcal{C}}(\mathfrak{b})\) is a natural number. For any \((\tilde{\mathfrak{c}}_1,\dots,\tilde{\mathfrak{c}}_r)\in I(\mathcal{C},\mathfrak{b})\), there exist \(i=1,\dots,r\) and \(\mathfrak{x}_i\in S_{\mathfrak{m}}(\tilde{\mathfrak{c}}_i)\) such that \(\mathfrak{x}_i\mathfrak{b}\in A_{\varphi_{\mathcal{C}}(\mathfrak{b})-1}\). Therefore \(\varphi_{\mathcal{C}}(\mathfrak{x}_i\mathfrak{b})<\varphi_{\mathcal{C}}(\mathfrak{b})\), and so \(\varphi_{\mathcal{C}}\) is a Euclidean function for \(\mathcal{C}\). Hence \(\mathcal{C}\) is Euclidean. Then, for any Euclidean function \(\varphi\) for \(\mathcal{C}\) and any natural number \(\mu\), we have
\begin{align*}
\{\mathfrak{b}\in \text{Int}(J_K^{\mathfrak{m}})~|~\varphi(\mathfrak{b})\le \mu\}\subset A_{\mathcal{C},\mu}=\{\mathfrak{b}\in \text{Int}(J_K^{\mathfrak{m}})~|~\varphi_{\mathcal{C}}(\mathfrak{b})\le \mu\}.
\end{align*}
Hence, for any \(\mathfrak{b}\in \text{Int}(J_K^{\mathfrak{m}})\), we have \(\varphi_{\mathcal{C}}(\mathfrak{b})\le \varphi(\mathfrak{b})\). Thus \(\varphi_{\mathcal{C}}\) is the minimal Euclidean function.
\\
\end{proof}

\begin{prop}\label{48} \n
Let \(\mathcal{C}\) be a generating set of the ray class group. If \(\exists C\in Cl_K^{\mathfrak{m}}\text{ such that }\text{Int}(C)\subset A_{\mathcal{C},\omega}\), then \(\mathcal{C}\) is Euclidean.
\end{prop}

\begin{proof}
For any \(i=1,\dots,r\), any \(\mathfrak{b}\in \text{Int}(CC_i)\), and any \((\tilde{\mathfrak{c}}_1,\dots,\tilde{\mathfrak{c}}_r)
\in I(\mathcal{C},\mathfrak{b})\), choose \(\mathfrak{x}_i\in S_{\mathfrak{m}}(\tilde{\mathfrak{c}}_i)\). Then \(\mathfrak{x}_i\mathfrak{b}\in \text{Int}(C)\subset A_{\mathcal{C},\omega}\). Hence \(\mathfrak{b}\in A_{\mathcal{C},\omega+1}=A_{\mathcal{C},\omega}\). Therefore \(\text{Int}(CC_i)\subset A_{\mathcal{C},\omega}\). Inductively, we obtain \(\text{Int}(J_K^{\mathfrak{m}})\subset A_{\mathcal{C},\omega}\). Thus \(\mathcal{C}\) is Euclidean.
\\
\end{proof}

We now prove Theorem 1.1.
\\

\begin{thm}[Reformulation of Theorem 1.1]\label{49} \n
Let \(\pi_{\mathfrak{m}}:J_K^{\mathfrak{m}}\to Cl_K^{\mathfrak{m}}\) be the natural projection.\\
(i) For any \(C\in \pi_{\mathfrak{m}}(A_{\mathcal{C},\omega})\), we have
\[
C=\prod_{i=1}^r C_i^{n_i}~~~\left(\sum_{i=1}^r n_i\le\min\{\varphi_{\mathcal{C}}(\mathfrak{b})~|~(1)\neq \mathfrak{b}\in C\cap A_{\mathcal{C},\omega}\}\right),
\]
and in particular,
\[
\left<\pi_{\mathfrak{m}}(A_{\mathcal{C},\omega})\right> \subset\left<\mathcal{C}\right>.
\]
In particular, if \(\mathcal{C}\) is Euclidean, then \(\mathcal{C}\) is a generating set of the ray class group.
(ii) If all of
\[
A_{C_i,1}~~~(i=1,\dots,r)
\]
are non-trivial, then
\[
\left<\pi_{\mathfrak{m}}(A_{\mathcal{C},\omega})\right>=\left<\mathcal{C}\right>.
\]
\end{thm}

\begin{proof}
(i) Let \(C\in \pi_{\mathfrak{m}}(A_{\mathcal{C},\omega}\cap J_K^{\mathfrak{m}})\). Choose an integral ideal \(\mathfrak{b}_0\) such that
\[
\varphi_{\mathcal{C}}(\mathfrak{b}_0)=\min\{\varphi_{\mathcal{C}}(\mathfrak{b})~|~(1)\neq \mathfrak{b}\in C\cap A_{\mathcal{C},\omega}\}.
\]
Choose \((\tilde{\mathfrak{c}}_1,\dots,\tilde{\mathfrak{c}}_r)\in I(\mathcal{C},\mathfrak{b}_0)\). Then there exist \(i=1,\dots,r\) and \(\mathfrak{x}_i\in S_{\mathfrak{m}}(\tilde{\mathfrak{c}}_i)\) such that \(\mathfrak{x}_i\mathfrak{b}_0\in A_{\mathcal{C},\varphi_{\mathcal{C}}(\mathfrak{b}_0)-1}\). Moreover, \([\mathfrak{x}_i\mathfrak{b}_0]=C_i^{-1}C\). This operation can be carried out inductively as long as \(\mathfrak{x}_i\mathfrak{b}_0\neq (1)\). Since the value of \(\varphi_{\mathcal{C}}\) strictly decreases, after at most \(\varphi_{\mathcal{C}}(\mathfrak{b}_0)\) steps we reach \((1)\). Hence there exist \(n_i\in\mathbb{N}\) with
\[
0\le \sum_{i=1}^r n_i\le\varphi_{\mathcal{C}}(\mathfrak{b}_0)
\]
such that \(\ds C\prod_{i=1}^r C_i^{-n_i}=[(1)]\). Therefore \(C\in \left<\mathcal{C}\right>\), and hence \(\pi_{\mathfrak{m}}(A_{\mathcal{C},\omega})\subset\left<\mathcal{C}\right>\).

(ii) Suppose that all of \(A_{C_i,1}~~(i=1,\dots,r)\) are non-trivial. For each \(i=1,\dots,r\), choose \(\mathfrak{b}_i\in A_{C_i,1}-A_{C_i,0}\). Then \([\mathfrak{b}_i]=C_i\). Therefore
\[
\left<\mathcal{C}\right>\subset\left<\pi_{\mathfrak{m}}(A_{\mathcal{C},\omega})\right>.
\]
This proves the assertion.
\\
\end{proof}

\section{A sufficient criterion for the existence of Euclidean systems}

Throughout this section, we assume GRH. 

For a finite set \(S\) of rational primes, we write
\[
K(\zeta_S)=K(\zeta_\ell~|~\ell\in S).
\]

We recall the form of Lenstra's theorem which will be used below. We state it in essentially the notation of \cite{Lenstra1977}.

Let \(K\) be a global field, let \(F/K\) be a finite Galois extension, and let
\[
C\subset \text{Gal}(F/K)
\]
be a union of conjugacy classes. Let \(W\subset K^\ast\) be a finitely generated infinite subgroup, and let \(k\) be a positive integer prime to the characteristic of \(K\). Let \(M=M(K,F,C,W,k)\) be the set of primes \(\mathfrak{p}\) of \(K\) such that
\[
(\mathfrak{p},F/K)\in C,~~~\text{ord}_{\mathfrak{p}}(w)=0~~(\text{for all }w\in W),
\]
and the index of the image of \(W\) in \(\kappa_{\mathfrak{p}}^{\ast}\) divides \(k\).

For a rational prime \(\ell\) prime to the characteristic of \(K\), let \(q(\ell)\) be the smallest power of \(\ell\) which does not divide \(k\), and put
\[
L_\ell=K(\zeta_{q(\ell)},\sqrt[q(\ell)]{W}).
\]
Let \(P_1\) be the set of all rational primes \(\ell\) such that
\[
W\subset (K^\ast)^{q(\ell)}.
\]
Put
\[
P_2=\left\{\ell~\middle|~L_{\ell}\subset F(\zeta_{P_1})\right\}.
\]
\\

\begin{thm}[Lenstra {\cite[Theorem 4.6]{Lenstra1977}}]\label{51} \n
With the notation above, if \(M\) is infinite, then there exists
\[
\tilde{\sigma}\in \text{Gal}(F(\zeta_{P_1})/K)
\]
such that
\[
\tilde{\sigma}|_F\in C
\]
and
\[
\tilde{\sigma}|_{L_\ell}\neq \text{id}_{L_\ell}
\]
for every rational prime \(\ell\in P_2\). Conversely, if such a \(\tilde{\sigma}\) exists and GRH holds, then \(M\) is infinite and has positive density.
\\
\end{thm}

We apply this theorem in the following special case. We assume that \(E^{\mathfrak{m}}\) is infinite; equivalently, \(K\) is neither \(\mathbb{Q}\) nor an imaginary quadratic field. Let \(H_K^{\mathfrak{m}}\) be the ray class field of \(K\) modulo \(\mathfrak{m}\).

We take
\[
F=H_K^{\mathfrak{m}},~~W=E^{\mathfrak{m}},~~k=1,
\]
and for a ray class \(C\in Cl_K^{\mathfrak{m}}\), we take the conjugacy class in \(\text{Gal}(H_K^{\mathfrak{m}}/K)\) to be the singleton consisting of the Artin element corresponding to \(C\). Since \(k=1\), we have
\[
q(\ell)=\ell,~~~L_\ell=K(\zeta_\ell,\sqrt[\ell]{E^{\mathfrak{m}}}).
\]
Thus the set \(P_1\) in Lenstra's theorem is 
\[
\{\ell\mid E^{\mathfrak m}\subset (K^\ast)^\ell\},
\]
and the set \(P_2\) is
\[
\{\ell~|~L_{\ell}\subset H_K^{\mathfrak{m}}(\zeta_{P_1})\}.
\]
\\

\begin{defn}\label{52} \n
Let \(C\in Cl_K^{\mathfrak m}\), and let
\[
\sigma_C\in \text{Gal}(H_K^{\mathfrak m}/K)
\]
be the Artin element corresponding to \(C\). We say that \(C\) is \textit{admissible} if there exists
\[
\tilde{\sigma}_C\in\text{Gal}(H_K^{\mathfrak{m}}(\zeta_{P_1})/K)
\]
such that
\[
\tilde{\sigma}_C|_{H_K^{\mathfrak{m}}}=\sigma_C
\]
and
\[
\tilde{\sigma}_C\notin\bigcup_{\ell\in P_2}\text{Gal}\left(H_K^{\mathfrak m}(\zeta_{P_1})/K(\zeta_\ell,\sqrt[\ell]{E^{\mathfrak{m}}})
\right).
\]
\\
\end{defn}

We use the following form of Lenstra's theorem.
\\

\begin{cor}\label{53} \n
Assume GRH. Let \(C\in Cl_K^{\mathfrak{m}}\), and let \(\sigma_C\in \text{Gal}(H_K^{\mathfrak{m}}/K)\) be the element corresponding to \(C\) under the Artin isomorphism. Then the set of prime ideals \(\mathfrak{p}\in C\) for which the reduction map
\[
E^{\mathfrak{m}}\longrightarrow \kappa_{\mathfrak{p}}^{\ast}
\]
is surjective has positive density if and only if \(C\) is admissible.
\\
\end{cor}

\begin{cor}
Let \(\mathcal{C}=\{C_1,\dots,C_r\}\subset Cl_K^{\mathfrak{m}}\) be a subset of ray classes. If each \(C_i\) is admissible, then
\[
\left<\pi_{\mathfrak{m}}(A_{\mathcal C,\omega})\right>=\left<\mathcal{C}\right>.
\]
\end{cor}

\begin{proof}
If \(C_i\) is admissible, then, by Corollary \ref{53}, there exists a prime ideal \(\mathfrak{p}_i\in C_i\) such that the reduction map \(E^{\mathfrak{m}}\longrightarrow \kappa_{\mathfrak{p}_i}^{\ast}\) is surjective. By Proposition \ref{45}, this implies \(\mathfrak{p}_i\in A_{C_i,1}-A_{C_i,0}\). Hence \(A_{C_i,1}-A_{C_i,0}\neq\emptyset\) for every \(i\). The assertion therefore follows by Theorem \ref{49}(ii).
\\
\end{proof}

The following lemma is the key ingredient in the proof of Theorem 1.2.
\\

\begin{lem}\label{55} \n
Let \(C\in Cl_K^{\mathfrak{m}}\) be admissible, and let \(C'\in Cl_K^{\mathfrak{m}}\). Let \(\mathfrak{p}\in CC'\) be a prime ideal which does not divide \(\mathfrak{m}\), any rational prime in \(P_1\), or \(\#\mu_K\). Then, for every \(\tilde{\mathfrak{c}}\in I(C',\mathfrak{p})\), there exists \(\mathfrak{x}\in S_{\mathfrak{m}}(\tilde{\mathfrak{c}})\) such that
\[
\mathfrak{x}\mathfrak{p}\in A_{C,1}.
\]
\end{lem}

\begin{proof}
Let \(\ds \tilde{\mathfrak{c}}=\frac{\tilde{\mathfrak{a}}}{\mathfrak{p}}\in I(C',\mathfrak{p})\). Then \(\tilde{\mathfrak{a}}\in \text{Int}(\mathfrak{p} {C'}^{-1})=\text{Int}(C)\). Let \(\sigma_C\in \text{Gal}(H_K^{\mathfrak{m}}/K)\) be the element corresponding to \(C\). Put \(D=\tilde{\mathfrak{a}}P_K^{\mathfrak{p}\mathfrak{m}}\in Cl_K^{\mathfrak{p}\mathfrak{m}}\), and let \(\tau_D\in \text{Gal}(H_K^{\mathfrak{p}\mathfrak{m}}/K)\) be the element corresponding to \(D\).

Since \(C\) is admissible, there exists an extension \(\tilde{\sigma}_C\in \text{Gal}(H_K^{\mathfrak{m}}(\zeta_{P_1})/K)\) of \(\sigma_C\) such that
\[
\tilde{\sigma}_C\notin \bigcup_{\ell\in P_2}\text{Gal}(H_K^{\mathfrak{m}}(\zeta_{P_1})/K(\zeta_\ell,\sqrt[\ell]{E^{\mathfrak{m}}})).
\]
Here \(H_K^{\mathfrak{m}}\subset H_K^{\mathfrak{m}}(\zeta_{P_1})\cap H_K^{\mathfrak{p}\mathfrak{m}}\). Since \(\mathfrak{p}\) is prime to \(\mathfrak{m}\) and to all rational primes in \(P_1\), the prime \(\mathfrak{p}\) is unramified in \(H_K^{\mathfrak{m}}(\zeta_{P_1})/K\). On the other hand, the extension \(H_K^{\mathfrak{p}\mathfrak{m}}/H_K^{\mathfrak{m}}\) is ramified only at primes above \(\mathfrak{p}\). Hence \(H_K^{\mathfrak{m}}(\zeta_{P_1})\cap H_K^{\mathfrak{p}\mathfrak{m}}\subset H_K^{\mathfrak{m}}\). Therefore
\[
H_K^{\mathfrak{m}}(\zeta_{P_1})\cap H_K^{\mathfrak{p}\mathfrak{m}}=H_K^{\mathfrak{m}}.
\]
By the compositum theorem, there exists \(\tilde{\tau}_D\in \text{Gal}(H_K^{\mathfrak{p}\mathfrak{m}}(\zeta_{P_1})/K)\) such that
\[
\tilde{\tau}_D|_{H_K^{\mathfrak{p}\mathfrak{m}}}=\tau_D,~~~\tilde{\tau}_D|_{H_K^{\mathfrak{m}}(\zeta_{P_1})}=\tilde{\sigma}_C.
\]
It follows that
\[
\tilde{\tau}_D\notin \bigcup_{\ell\in P_2}\text{Gal}(H_K^{\mathfrak{p}\mathfrak{m}}(\zeta_{P_1})/K(\zeta_\ell,\sqrt[\ell]{E^{\mathfrak{m}}})).
\]

We next show that no new obstruction occurs. Suppose, to the contrary, that there exists \(\ell\notin P_2\) such that \(K(\zeta_\ell,\sqrt[\ell]{E^{\mathfrak{m}}})\subset H_K^{\mathfrak{p}\mathfrak{m}}(\zeta_{P_1})\). Since \(H_K^{\mathfrak{p}\mathfrak{m}}(\zeta_{P_1})/K\) is abelian, the extension \(K(\zeta_\ell,\sqrt[\ell]{E^{\mathfrak{m}}})/K\) is abelian. Moreover,
\[
H_K^{\mathfrak{m}}(\zeta_{P_1})\subsetneq H_K^{\mathfrak{m}}(\zeta_{P_1})K(\zeta_\ell,\sqrt[\ell]{E^{\mathfrak{m}}})\subset H_K^{\mathfrak{p}\mathfrak{m}}(\zeta_{P_1}).
\]
Since \(\mathfrak{p}\) is prime to \(\mathfrak{m}\) and to all rational primes in \(P_1\), the extension \(H_K^{\mathfrak{m}}(\zeta_{P_1})
K(\zeta_\ell,\sqrt[\ell]{E^{\mathfrak{m}}})/H_K^{\mathfrak{m}}(\zeta_{P_1})\) is ramified at \(\mathfrak{p}\). Hence \(K(\zeta_\ell,\sqrt[\ell]{E^{\mathfrak{m}}})/K\) is ramified at \(\mathfrak{p}\). Since this extension is unramified outside primes above \(\ell\), we must have \(\mathfrak{p}|\ell\). Since \(\mathfrak{p}\) is prime to \(\#\mu_K\), we have \(\ell\nmid \#\mu_K\), and therefore \(\zeta_\ell\notin K\).

Also, since \(\ell\notin P_1\), we have \(E^{\mathfrak{m}}\not\subset (K^\ast)^\ell\). Consider the natural \(\mathbb{F}_\ell\)-linear map
\[
K^\ast/(K^\ast)^\ell\longrightarrow K(\zeta_\ell)^\ast/(K(\zeta_\ell)^\ast)^\ell.
\]
The composite with the norm map
\[
N_{K(\zeta_\ell)/K}:K(\zeta_\ell)^\ast/(K(\zeta_\ell)^\ast)^\ell\longrightarrow K^\ast/(K^\ast)^\ell
\]
is multiplication by \([K(\zeta_\ell):K]\). Since \(\ell\nmid [K(\zeta_\ell):K]\), this composite is an automorphism. Hence
\[
(K^\ast)^\ell=(K(\zeta_\ell)^\ast)^\ell\cap K^\ast.
\]
It follows that \(E^{\mathfrak{m}}\not\subset (K(\zeta_\ell)^\ast)^\ell\). Thus \(K(\zeta_\ell,\sqrt[\ell]{E^{\mathfrak{m}}})/K(\zeta_\ell)\) is a non-trivial \(\ell\)-Kummer extension.

Choose \(\varepsilon\in E^{\mathfrak{m}}-(K(\zeta_\ell)^\ast)^\ell\), a non-trivial element \(1\neq \delta\in \text{Gal}(K(\zeta_\ell)/K)\), and \(1\neq \sigma\in\text{Gal}(K(\zeta_\ell,\sqrt[\ell]{E^{\mathfrak{m}}})/K(\zeta_\ell))\) such that
\[
\sigma\notin\text{Gal}(K(\zeta_\ell,\sqrt[\ell]{E^{\mathfrak{m}}})/K(\zeta_\ell,\sqrt[\ell]{\varepsilon})).
\]
Let \(\tilde{\delta}\in \text{Gal}(K(\zeta_\ell,\sqrt[\ell]{E^{\mathfrak{m}}})/K)\) be an extension of \(\delta\). Since \(K(\zeta_\ell,\sqrt[\ell]{E^{\mathfrak{m}}})/K\) is abelian, we have \(\tilde{\delta}\sigma\tilde{\delta}^{-1}=\sigma\). Now \(\ds \frac{\sigma\sqrt[\ell]{\varepsilon}}{\sqrt[\ell]{\varepsilon}}\) is a primitive \(\ell\)-th root of unity. Moreover,
\begin{align*}
\frac{\sigma\sqrt[\ell]{\varepsilon}}{\sqrt[\ell]{\varepsilon}}=\frac{(\tilde{\delta}\sigma\tilde{\delta}^{-1})\sqrt[\ell]{\varepsilon}}
{\sqrt[\ell]{\varepsilon}}=\tilde{\delta}\left(\frac{\sigma(\tilde{\delta}^{-1}\sqrt[\ell]{\varepsilon})}{\tilde{\delta}^{-1}\sqrt[\ell]{\varepsilon}}
\right)=\delta\left(\frac{\sigma\sqrt[\ell]{\varepsilon}}{\sqrt[\ell]{\varepsilon}}\right).
\end{align*}
This contradicts the fact that \(\delta\neq 1\) acts non-trivially on \(\mu_\ell\). Hence no such \(\ell\notin P_2\) exists.

Therefore
\[
\tilde{\tau}_D\notin\bigcup_{K(\zeta_\ell,\sqrt[\ell]{E^{\mathfrak{m}}})\subset H_K^{\mathfrak{p}\mathfrak{m}}(\zeta_{P_1})}\text{Gal}(H_K^{\mathfrak{p}\mathfrak{m}}(\zeta_{P_1})/K(\zeta_\ell,\sqrt[\ell]{E^{\mathfrak{m}}})).
\]
Thus the class \(D\in Cl_K^{\mathfrak{p}\mathfrak{m}}\) satisfies Lenstra's obstruction condition for the group \(E^{\mathfrak{m}}\). By Lenstra's theorem, there exist prime ideals \(\mathfrak{q}\in D\) of positive density such that the reduction map \(E^{\mathfrak m}\longrightarrow \kappa_{\mathfrak{q}}^{\ast}\) is surjective.

Choose such a prime ideal \(\mathfrak{q}\in D\). Since \(D=\tilde{\mathfrak{a}}P_K^{\mathfrak{p}\mathfrak{m}}\), we have
\[
\frac{\mathfrak{q}}{\mathfrak{p}}\in \mathfrak{p}^{-1}\text{Int}(\tilde{\mathfrak{a}}P_K^{\mathfrak{p}\mathfrak{m}})=S_{\mathfrak{m}}(\tilde{\mathfrak{c}}).
\]
Put \(\ds \mathfrak{x}=\frac{\mathfrak{q}}{\mathfrak{p}}\), then \(\mathfrak{x}\mathfrak{p}=\mathfrak{q}\). Moreover, since \(\mathfrak{q}\in D\) and \(D\) maps to \(C\) modulo \(\mathfrak{m}\), we have \(\mathfrak{q}\in C\). Together with the surjectivity of \(E^{\mathfrak{m}}\longrightarrow \kappa_{\mathfrak{q}}^{\ast}\), the description of the first Motzkin layer gives \(\mathfrak{q}\in A_{C,1}\).
Therefore \(\mathfrak{x}\mathfrak{p}\in A_{C,1}\). This proves the lemma.
\\
\end{proof}

\begin{cor}\label{56} \n
Suppose that \(C_1\) is admissible. \\
(i) We have
\[
\left\{\mathfrak{p}\in \text{Spec}(\ded_K)\cap\bigcup_{i=1}^r C_1C_i~\middle|~\mathfrak{p} \text{ is prime to every rational prime in }P_1\text{ and to }\#\mu_K\right\}\subset A_{\mathcal{C},2}.
\]
(ii) For every \(k\ge 3\), we have
\[
\text{Spec}(\ded_K)\cap\bigcup_{i_1,\dots,i_{k-1}=1}^rC_1C_{i_1}\cdots C_{i_{k-1}}\subset A_{\mathcal{C},k}.
\]
(iii) Put
\[
s=\min\left\{\tilde{s}~\middle|~\left<\mathcal{C}\right>=\left\{\prod_{i=1}^r C_i^{n_i}~\middle|~n_1\in\mathbb{N}_{>0},~n_2,\dots,n_r\in\mathbb{N},~0<\sum_{i=1}^r n_i\le \tilde{s}\right\}\right\}.
\]
Then
\[
\text{Spec}(\ded_K)\cap\pi_{\mathfrak{m}}^{-1}(\left<\mathcal{C}\right>)\subset A_{\mathcal{C},s+2}.
\]
\end{cor}

\begin{proof}
The assertion (i) follows immediately from Lemma \ref{55}.

We prove (ii). First let \(k=3\). Let \(\mathfrak{p}\in C_1C_{i_1}C_{i_2}\) be a prime ideal, and let
\[
(\tilde{\mathfrak{c}}_1,\dots,\tilde{\mathfrak{c}}_r)=\left(\frac{\tilde{\mathfrak{a}}_1}{\mathfrak{p}},\dots,\frac{\tilde{\mathfrak{a}}_r}{\mathfrak{p}}\right)\in I(\mathcal{C},\mathfrak{p}).
\]
Then \(\tilde{\mathfrak{a}}_{i_2}\in [\mathfrak{p}]C_{i_2}^{-1}=C_1C_{i_1}\). By the density theorem, we may choose a prime ideal \(\mathfrak{q}_{i_2}\in\tilde{\mathfrak{a}}_{i_2}P_K^{\mathfrak{p}\mathfrak{m}}\) which is prime to \(\mathfrak{m}\), to every rational prime in \(P_1\), and to \(\#\mu_K\). Then \(\ds \mathfrak{x}_{i_2}=\frac{\mathfrak{q}_{i_2}}{\mathfrak{p}}\in S_{\mathfrak{m}}(\tilde{\mathfrak{c}}_{i_2})\), and \(\mathfrak{x}_{i_2}\mathfrak{p}=\mathfrak{q}_{i_2}\in A_{\mathcal{C},2}\) by (i). Hence 
\[
\mathfrak{p}\in A_{\mathcal{C},3}.
\]

Now assume that the assertion holds for \(k-1\). Let \(\mathfrak{p}\in C_1C_{i_1}\cdots C_{i_{k-1}}\) be a prime ideal, and let \(\ds (\tilde{\mathfrak{c}}_1,\dots,\tilde{\mathfrak{c}}_r)=\left(\frac{\tilde{\mathfrak{a}}_1}{\mathfrak{p}},\dots,\frac{\tilde{\mathfrak{a}}_r}{\mathfrak{p}}\right)\in I(\mathcal{C},\mathfrak{p})\). Then \(\tilde{\mathfrak{a}}_{i_{k-1}}\in[\mathfrak{p}]C_{i_{k-1}}^{-1}=C_1C_{i_1}\cdots C_{i_{k-2}}\). By the density theorem, we may choose a prime ideal \(\mathfrak{q}_{i_{k-1}}\in\tilde{\mathfrak{a}}_{i_{k-1}}P_K^{\mathfrak{p}\mathfrak{m}}\). Then
\[
\mathfrak{x}_{i_{k-1}}=\frac{\mathfrak{q}_{i_{k-1}}}{\mathfrak{p}}\in S_{\mathfrak{m}}(\tilde{\mathfrak{c}}_{i_{k-1}}),
\]
and
\[
\mathfrak{x}_{i_{k-1}}\mathfrak{p}=\mathfrak{q}_{i_{k-1}}\in\text{Spec}(\ded_K)\cap C_1C_{i_1}\cdots C_{i_{k-2}}\subset A_{\mathcal{C},k-1}
\]
by the induction hypothesis. Hence
\[
\mathfrak{p}\in A_{\mathcal{C},k}.
\]
Thus (ii) follows by induction.

Finally, (iii) follows from (ii) and the definition of \(s\).
\\
\end{proof}

\begin{thm}\label{57} \n
Let \(\mathcal{C}=\{C_1,\dots,C_r\}\subset Cl_K^{\mathfrak{m}}\) be a subset of ray classes, and suppose that \(C_1\) is admissible. Let
\[
s=\min\left\{\tilde{s}~\middle|~\left<\mathcal C\right>=\left\{\prod_{i=1}^r C_i^{n_i}~\middle|~n_1\in\mathbb{N}_{>0},~n_2,\dots,n_r\in\mathbb{N},~0<\sum_{i=1}^r n_i\le \tilde{s}\right\}\right\}.
\]
Then
\[
\text{Spec}(\ded_K)\cap \pi_{\mathfrak{m}}^{-1}(\left<\mathcal{C}\right>)\subset A_{\mathcal{C},s+2}.
\]

Define
\[
\lambda:\{(1)\}\cup\left(\text{Spec}(\ded_K)\cap \pi_{\mathfrak{m}}^{-1}(\left<\mathcal{C}\right>)\right)\longrightarrow \{0,1,\dots,s+2\}
\]
by
\[
\lambda(\mathfrak{p})=\begin{cases}
0 & (\mathfrak{p}=(1)), \\ 1 & ((1)\neq \mathfrak{p}\in A_{C_1,1}), \\ 2 & \left(\mathfrak{p}\in \ds\bigcup_{i=1}^r C_1C_i\text{ and }\mathfrak{p}\text{ is prime to every rational prime in }P_1\text{ and to }\#\mu_K\right), \\ k & \left(
\begin{array}{l}
\text{otherwise, where }k\text{ is the minimum integer in }\{3,\dots,s+2\} \\
\text{such that there exist }n_1\in\mathbb{N}_{>0},~n_2,\dots,n_r\in\mathbb{N} \\
\text{with }\mathfrak{p}\in \ds\prod_{i=1}^r C_i^{n_i}\text{ and }\sum_{i=1}^r n_i=k
\end{array}
\right).
\end{cases}
\]
Define
\[
\theta:\left<\text{Spec}(\ded_K)\cap\pi_{\mathfrak{m}}^{-1}(\left<\mathcal{C}\right>)\right>\longrightarrow\mathbb{N}
\]
by
\[
\theta\left(\prod_{\mathfrak{p}}\mathfrak p^{\nu_{\mathfrak{p}}}\right)=\sum_{\mathfrak{p}}(\lambda(\mathfrak{p})+s+1)\nu_{\mathfrak{p}}.
\]
Then, for every non-trivial integral ideal
\[
(1)\neq\mathfrak{b}\in\left<\text{Spec}(\ded_K)\cap\pi_{\mathfrak{m}}^{-1}(\left<\mathcal{C}\right>)\right>,
\]
and for every
\[
(\tilde{\mathfrak{c}}_1,\dots,\tilde{\mathfrak{c}}_r)\in I(\mathcal{C},\mathfrak{b}),
\]
there exist \(i=1,\dots,r\) and \(\mathfrak{x}_i\in S_{\mathfrak{m}}(\tilde{\mathfrak{c}}_i)\) such that
\[
\theta(\mathfrak{x}_i\mathfrak{b})<\theta(\mathfrak{b}).
\]
Moreover,
\[
\left<\text{Spec}(\ded_K)\cap\pi_{\mathfrak{m}}^{-1}(\left<\mathcal{C}\right>)\right>\subset A_{\mathcal{C},\omega}\subset\pi_{\mathfrak{m}}^{-1}(\left<\mathcal{C}\right>).
\]
In particular, if \(\mathcal{C}\) generates \(Cl_K^{\mathfrak{m}}\), then
\[
A_{\mathcal{C},\omega}=\text{Int}(J_K^{\mathfrak{m}}),
\]
and \(\mathcal{C}\) is a Euclidean system with completely additive Euclidean function \(\theta\).
\end{thm}

\begin{proof}
Let \((1)\neq\mathfrak{b}\in\left<\text{Spec}(\ded_K)\cap\pi_{\mathfrak{m}}^{-1}(\left<\mathcal{C}\right>)\right>\), and let
\[
(\tilde{\mathfrak{c}}_1,\dots,\tilde{\mathfrak{c}}_r)=\left(\frac{\tilde{\mathfrak{a}}_1}{\tilde{\mathfrak{b}}_1},\dots,\frac{\tilde{\mathfrak{a}}_r}{\tilde{\mathfrak{b}}_r}\right)\in I(\mathcal{C},\mathfrak{b}).
\]

Suppose first that, for some \(i\), the denominator \(\tilde{\mathfrak{b}}_i\) is not a prime ideal. Then
\[
\theta(\tilde{\mathfrak{b}}_i)\ge 2s+4.
\]
By the density theorem, we may choose a prime ideal \(\mathfrak{p}_i\in \tilde{\mathfrak{a}}_iP_K^{\tilde{\mathfrak{b}}_i\mathfrak{m}}\). Then
\[
\mathfrak{x}_i=\frac{\mathfrak{p}_i}{\tilde{\mathfrak{b}}_i}
\in S_{\mathfrak{m}}(\tilde{\mathfrak{c}}_i),
\]
and
\[
\theta(\mathfrak{x}_i\mathfrak{b})=\theta(\mathfrak{p}_i)-\theta(\tilde{\mathfrak{b}}_i)+\theta(\mathfrak{b})\le(2s+3)-(2s+4)+\theta(\mathfrak{b})<\theta(\mathfrak{b}).
\]

We may therefore assume that all of \(\tilde{\mathfrak{b}}_1,\dots,\tilde{\mathfrak{b}}_r\) are prime ideals. Write
\[
(\tilde{\mathfrak{c}}_1,\dots,\tilde{\mathfrak{c}}_r)=\left(\prod_{j\neq i}\mathfrak{c}_j\right)_i\tilde{\mathfrak{c}},
\]
where
\[
\left(\prod_{j\neq i}\mathfrak{c}_j\right)_i\in D(\mathcal{C},\mathfrak{b}),~~~\tilde{\mathfrak{c}}\in I(C_1\cdots C_r,\mathfrak{b}).
\]
Then, for every \(i=2,\dots,r\), we have
\[
\frac{\tilde{\mathfrak{a}}_i}{\tilde{\mathfrak{b}}_i}=\frac{\tilde{\mathfrak{a}}_1\mathfrak{c}_1}{\tilde{\mathfrak{b}}_1\mathfrak{c}_i},
\]
where
\[
\tilde{\mathfrak{b}}_1\mid\mathfrak{b},~~~(\mathfrak{c}_1,\mathfrak{b})=1,
\]
and both \(\tilde{\mathfrak{b}}_1\) and \(\tilde{\mathfrak{b}}_i\) are prime ideals. Hence
\[
\tilde{\mathfrak{b}}_i=\tilde{\mathfrak{b}}_1
\]
for all \(i\). In what follows, we write this common prime denominator simply as \(\tilde{\mathfrak{b}}\).

Suppose that
\[
\lambda(\tilde{\mathfrak{b}})=k\ge 4.
\]
Then \(\ds \tilde{\mathfrak b}\in \prod_{i=1}^r C_i^{n_i}\) for some \(\ds n_1\in\mathbb{N}_{>0},~n_2,\dots,n_r\in\mathbb{N},~\sum_{i=1}^r n_i=k\). Since \(k\ge 4\), either there exists \(i=2,\dots,r\) such that \(n_i>0\), or \(n_1>1,~i=1\). Choose such an index \(i\). Then
\[
\tilde{\mathfrak{a}}_i\in[\tilde{\mathfrak{b}}]C_i^{-1}=C_i^{n_i-1}\prod_{j\neq i}C_j^{n_j}.
\]
By the density theorem, there exists a prime ideal
\[
\mathfrak{q}_i\in\text{Int}(\tilde{\mathfrak{a}}_iP_K^{\tilde{\mathfrak{b}}\mathfrak{m}})\subset C_i^{n_i-1}\prod_{j\neq i}C_j^{n_j}.
\]
Then \(\ds \lambda(\mathfrak{q}_i)\le n_i-1+\sum_{j\neq i}n_j<k\). Put
\[
\mathfrak{x}_i=\frac{\mathfrak{q}_i}{\tilde{\mathfrak{b}}}\in S_{\mathfrak{m}}(\tilde{\mathfrak{c}}_i).
\]
Since \(\theta(\mathfrak{q}_i)<\theta(\tilde{\mathfrak{b}})\), we get
\[
\theta(\mathfrak{x}_i\mathfrak{b})=\theta(\mathfrak{q}_i)-\theta(\tilde{\mathfrak{b}})+\theta(\mathfrak{b})<\theta(\mathfrak{b}).
\]

Suppose next that
\[
\lambda(\tilde{\mathfrak{b}})=3.
\]
Then \(\ds \tilde{\mathfrak{b}}\in \prod_{i=1}^r C_i^{n_i}\) for some
\[
n_1\in\mathbb{N}_{>0},~n_2,\dots,n_r\in\mathbb{N},~\sum_{i=1}^r n_i=3.
\]
Again, either there exists \(i=2,\dots,r\) such that \(n_i>0\), or \(n_1>1,~i=1\). Choose such an index \(i\). Then
\[
\tilde{\mathfrak{a}}_i\in [\tilde{\mathfrak{b}}]C_i^{-1}=C_i^{n_i-1}\prod_{j\neq i}C_j^{n_j}\subset \bigcup_{j=1}^r C_1C_j.
\]
By the density theorem, we may choose a prime ideal \(\mathfrak{q}_i\in\text{Int}(\tilde{\mathfrak{a}}_iP_K^{\tilde{\mathfrak{b}}\mathfrak{m}})\) which is prime to every rational prime in \(P_1\) and to \(\#\mu_K\). Then \(\lambda(\mathfrak{q}_i)=2\). Put
\[
\mathfrak{x}_i=\frac{\mathfrak{q}_i}{\tilde{\mathfrak{b}}}\in S_{\mathfrak{m}}(\tilde{\mathfrak{c}}_i).
\]
Since \(\theta(\mathfrak{q}_i)<\theta(\tilde{\mathfrak{b}})\), we get
\[
\theta(\mathfrak{x}_i\mathfrak{b})<\theta(\mathfrak{b}).
\]

Suppose next that
\[
\lambda(\tilde{\mathfrak{b}})=2.
\]
Then there exists \(i=1,\dots,r\) such that \(\tilde{\mathfrak{b}}\in C_1C_i\). By Lemma \ref{55}, there exists
\[
\mathfrak{x}_i\in S_{\mathfrak{m}}(\tilde{\mathfrak{c}}_i)
\]
such that
\[
\mathfrak{x}_i\tilde{\mathfrak{b}}\in A_{C_1,1}.
\]
Hence \(\lambda(\mathfrak{x}_i\tilde{\mathfrak{b}})\le 1\), and therefore
\[
\theta(\mathfrak{x}_i\mathfrak{b})=\theta(\mathfrak{x}_i\tilde{\mathfrak{b}})-\theta(\tilde{\mathfrak{b}})+\theta(\mathfrak{b})<\theta(\mathfrak{b}).
\]

Finally, suppose that
\[
\lambda(\tilde{\mathfrak{b}})=1.
\]
Then \(\tilde{\mathfrak{b}}\in A_{C_1,1}\). Hence there exists \(\mathfrak{x}_1\in S_{\mathfrak{m}}(\tilde{\mathfrak{c}}_1)\) such that
\[
\mathfrak{x}_1\tilde{\mathfrak{b}}=(1).
\]
Thus
\[
\theta(\mathfrak{x}_1\mathfrak{b})=\theta(\mathfrak{x}_1\tilde{\mathfrak{b}})-\theta(\tilde{\mathfrak{b}})+\theta(\mathfrak{b})<\theta(\mathfrak{b}).
\]

We have proved that, for every non-trivial integral ideal
\[
\mathfrak{b}\in\left<\text{Spec}(\ded_K)\cap\pi_{\mathfrak{m}}^{-1}(\left<\mathcal{C}\right>)\right>,
\]
and every \((\tilde{\mathfrak{c}}_1,\dots,\tilde{\mathfrak{c}}_r)\in I(\mathcal{C},\mathfrak{b})\), there exist \(i=1,\dots,r\) and \(\mathfrak{x}_i\in S_{\mathfrak{m}}(\tilde{\mathfrak{c}}_i)\) such that
\[
\theta(\mathfrak{x}_i\mathfrak{b})<\theta(\mathfrak{b}).
\]

For \(n\in\mathbb{N}\), put
\[
\tilde{A}_n=\left\{\mathfrak{b}\in\left<\text{Spec}(\ded_K)\cap\pi_{\mathfrak{m}}^{-1}(\left<\mathcal{C}\right>)\right>~\middle|~\theta(\mathfrak{b})\le n\right\}.
\]
Then \(\tilde{A}_0={(1)}=A_{\mathcal{C},0}\). Assume that \(\tilde{A}_{n-1}\subset A_{\mathcal{C},n-1}\). Let \(\mathfrak{b}\in \tilde{A}_n\). For every \((\tilde{\mathfrak{c}}_1,\dots,\tilde{\mathfrak{c}}_r)\in I(\mathcal{C},\mathfrak{b})\), there exist \(i=1,\dots,r\) and \(\mathfrak{x}_i\in S_{\mathfrak{m}}(\tilde{\mathfrak{c}}_i)\) such that
\[
\theta(\mathfrak{x}_i\mathfrak{b})<\theta(\mathfrak{b})\le n.
\]
Hence \(\mathfrak{x}_i\mathfrak{b}\in \tilde{A}_{n-1}\subset A_{\mathcal{C},n-1}\), and therefore \(\mathfrak{b}\in A_{\mathcal{C},n}\). Thus
\[
\tilde{A}_n\subset A_{\mathcal{C},n}.
\]
By induction, this holds for all \(n\in\mathbb{N}\). Consequently,
\[
\left<\text{Spec}(\ded_K)\cap\pi_{\mathfrak{m}}^{-1}(\left<\mathcal{C}\right>)\right>=\bigcup_{n=0}^{\infty}\tilde{A}_n\subset\bigcup_{n=0}^{\infty}A_{\mathcal{C},n}=A_{\mathcal{C},\omega}.
\]
The inclusion
\[
A_{\mathcal{C},\omega}\subset\pi_{\mathfrak{m}}^{-1}(\left<\mathcal{C}\right>)
\]
follows from (i) in Theorem \ref{49}. If \(\mathcal{C}\) generates \(Cl_K^{\mathfrak{m}}\), then \(\pi_{\mathfrak{m}}^{-1}(\left<\mathcal{C}\right>)=\text{Int}(J_K^{\mathfrak{m}})\), and hence
\[
A_{\mathcal{C},\omega}=\text{Int}(J_K^{\mathfrak{m}}).
\]
Therefore \(\mathcal{C}\) is a Euclidean system, and \(\theta\) is a completely additive Euclidean function for \(\mathcal{C}\).
\\
\end{proof}

\begin{thm}[Reformulation of Theorem 1.2]\label{58} \n
Assume GRH. Let \(K\) be a totally real Galois number field of degree \(n=[K:\mathbb Q]\ge 3\). Let \(p\) be an odd rational prime which does not split completely in \(K\). Let \(N\in\mathbb{N}_{>0},~\mathfrak{m}=(p)^N\). Let \(\mathcal{C}=\{C_1,\dots,C_r\}\subset Cl_K^{\mathfrak{m}}\) be a generating set of the ray class group. Then at least one of \(C_1,\dots,C_r\) is admissible. In particular, \(\mathcal{C}\) is a Euclidean system.
\end{thm}

\begin{proof}
We first prove that
\[
P_1\subset \{p\},~~P_2\subset \{2,p\}.
\]
Let \(\mathfrak{p}_1,\dots,\mathfrak{p}_g\) be the prime ideals of \(K\) above \(p\). Since \(K/\mathbb{Q}\) is Galois and \(p\) does not split completely in \(K\), we have \(\ds g\le \frac{n}{2}\).

Let \(\ell\in P_1\). Since \(E^{\mathfrak{m}}\subset E_K^\ell\), we have
\[
E_K/E^{\mathfrak{m}}\twoheadrightarrow E_K/E_K^\ell.
\]
On the other hand, reduction modulo \(\mathfrak{m}\) gives
\[
E_K/E^{\mathfrak{m}}\hookrightarrow (\ded_K/\mathfrak{m})^\ast\cong\prod_{i=1}^g\left(\kappa_{\mathfrak{p}_i}^{\ast}\oplus U_{\mathfrak{p}_i}^{(1)}/U_{\mathfrak{p}_i}^{(e_p^KN)}
\right).
\]
Moreover,
\[
E_K/E_K^\ell\cong\begin{cases}
(\mathbb{Z}/\ell\mathbb{Z})^{n-1} & (\ell\neq 2) \\ (\mathbb{Z}/2\mathbb{Z})^n & (\ell=2)
\end{cases}.
\]
Suppose that \(\ell\neq p\). Then the \(\ell\)-rank of \(\ds \prod_{i=1}^g\left(\kappa_{\mathfrak p_i}^{\ast}\oplus U_{\mathfrak{p}_i}^{(1)}/U_{\mathfrak{p}_i}^{(e_p^KN)}\right)\) is at most \(g\). Hence
\[
n-1\le \text{rank}_{\ell}(E_K/E_K^\ell)\le \text{rank}_{\ell}(E_K/E^{\mathfrak{m}})\le \text{rank}_{\ell}\left(\prod_{i=1}^g\left(\kappa_{\mathfrak{p}_i}^{\ast}\oplus U_{\mathfrak{p}_i}^{(1)}/U_{\mathfrak{p}_i}^{(e_p^KN)}\right)\right)\le g\le \frac{n}{2}.
\]
This contradicts \(n\ge 3\). Therefore \(P_1\subset\{p\}\).

Next let
\[
\ell\in P_2\setminus P_1.
\]
Then \(K(\zeta_\ell,\sqrt[\ell]{E^{\mathfrak{m}}})\subset H_K^{\mathfrak{m}}(\zeta_{P_1})\). Since \(H_K^{\mathfrak{m}}(\zeta_{P_1})/K\) is abelian, the extension \(K(\zeta_\ell,\sqrt[\ell]{E^{\mathfrak{m}}})/K\) is also abelian. Hence either \(\zeta_\ell\in K\) or \(\sqrt[\ell]{E^{\mathfrak{m}}}\subset K\). Since \(\ell\notin P_1\), we have \(\sqrt[\ell]{E^{\mathfrak{m}}}\not\subset K\). Thus \(\zeta_\ell\in K\). Because \(K\) is totally real, this implies \(\ell=2\). Therefore
\[
P_2\setminus P_1\subset\{2\},
\]
and hence \(P_2\subset\{2,p\}\).

If \(P_2=\emptyset\), then every ray class is admissible, and there is nothing to prove.

Suppose first that
\[
P_1=\emptyset,~~P_2=\{2\}.
\]
Then \(K(\sqrt{E^{\mathfrak{m}}})\subset H_K^{\mathfrak{m}}\). Since \(2\notin P_1\), we have \(K\subsetneq K(\sqrt{E^{\mathfrak{m}}})
\subset H_K^{\mathfrak{m}}\). Since \(C_1,\dots,C_r\) generate \(Cl_K^{\mathfrak m}\), there exists \(i=1,\dots,r\) such that, if
\[
\sigma_i\in\operatorname{Gal}(H_K^{\mathfrak{m}}/K)
\]
is the Artin element corresponding to \(C_i\), then
\[
\sigma_i\notin\text{Gal}(H_K^{\mathfrak{m}}/K(\sqrt{E^{\mathfrak{m}}})).
\]
Therefore \(C_i\) is admissible.

Suppose next that
\[
P_1=P_2=\{p\}.
\]
Let \(\sigma_1\in\text{Gal}(H_K^{\mathfrak m}/K)\) be the Artin element corresponding to \(C_1\), and choose an extension \(\tilde{\sigma}_1\in\text{Gal}(H_K^{\mathfrak m}(\zeta_p)/K)\) of \(\sigma_1\). If \(\tilde{\sigma}_1\notin\text{Gal}(H_K^{\mathfrak{m}}(\zeta_p)/K(\zeta_p))\), then \(C_1\) is admissible. Otherwise, \(\tilde{\sigma}_1\in\text{Gal}(H_K^{\mathfrak{m}}(\zeta_p)/K(\zeta_p))\). Since \(K\) is totally real, the ray class field \(H_K^{\mathfrak{m}}\) is also totally real. Hence there exists \(1\neq J\in\text{Gal}(H_K^{\mathfrak{m}}(\zeta_p)/H_K^{\mathfrak{m}}(\zeta_p+\zeta_p^{-1}))\) such that \(J(\zeta_p)=\zeta_p^{-1}\). Then
\[
J\tilde{\sigma}_1\notin\text{Gal}(H_K^{\mathfrak{m}}(\zeta_p)/K(\zeta_p)),
\]
and
\[
(J\tilde{\sigma}_1)|_{H_K^{\mathfrak{m}}}=\sigma_1.
\]
Thus \(C_1\) is admissible.

It remains to consider the case
\[
P_1=\{p\},~~P_2=\{2,p\}.
\]
Suppose, to the contrary, that \(K(\sqrt{E^{\mathfrak m}})\cap H_K^{\mathfrak m}=K\). As above, reduction modulo \(\mathfrak{m}\) gives
\[
E_K/E^{\mathfrak m}\hookrightarrow (\ded_K/\mathfrak{m})^\ast\cong\prod_{i=1}^g\left(\kappa_{\mathfrak{p}_i}^{\ast}\oplus U_{\mathfrak{p}_i}^{(1)}/U_{\mathfrak{p}_i}^{(e_p^KN)}\right),
\]
and
\[
E_K/E_K^2\cong (\mathbb{Z}/2\mathbb{Z})^n.
\]
Therefore
\begin{align*}
\text{rank}_2\text{Gal}(K(\sqrt{E^{\mathfrak{m}}})/K)&=\text{rank}_2(E^{\mathfrak{m}}E_K^2/E_K^2) \\
&\ge\text{rank}_2(E_K/E_K^2)-\text{rank}_2(E_K/E^{\mathfrak{m}})\ge n-g\ge\frac{n}{2}>1
\end{align*}
By the compositum theorem, we obtain
\begin{align*}
1=\text{rank}_2\text{Gal}(H_K^{\mathfrak{m}}(\zeta_p)/H_K^{\mathfrak{m}})\ge\text{rank}_2\text{Gal}(H_K^{\mathfrak{m}}(\sqrt{E^{\mathfrak{m}}})/H_K^{\mathfrak{m}})=\text{rank}_2\text{Gal}(K(\sqrt{E^{\mathfrak{m}}})/K)>1,
\end{align*}
which is a contradiction. Hence
\[
K\subsetneq K(\sqrt{E^{\mathfrak{m}}})\cap H_K^{\mathfrak{m}}\subset H_K^{\mathfrak{m}}.
\]

Since \(C_1,\dots,C_r\) generate \(Cl_K^{\mathfrak{m}}\), there exists \(i=1,\dots,r\) such that, if
\[
\sigma_i\in\text{Gal}(H_K^{\mathfrak{m}}/K)
\]
is the Artin element corresponding to \(C_i\), then
\[
\sigma_i\notin\text{Gal}(H_K^{\mathfrak{m}}/H_K^{\mathfrak{m}}\cap K(\sqrt{E^{\mathfrak{m}}})).
\]
Choose an extension \(\tilde{\sigma}_i\in\text{Gal}(H_K^{\mathfrak{m}}(\zeta_p)/K)\) of \(\sigma_i\). If
\[
\tilde{\sigma}_i\notin\text{Gal}(H_K^{\mathfrak{m}}(\zeta_p)/K(\zeta_p)),
\]
then 
\[
\tilde{\sigma}_i|_{K(\zeta_p)}\neq \id,~~\tilde{\sigma}_i|_{K(\sqrt{E^{\mathfrak{m}}})}\neq \id,~~\tilde{\sigma}_i|_{H_K^{\mathfrak{m}}}=\sigma_i.
\]
Thus \(C_i\) is admissible.

Otherwise, \(\tilde{\sigma}_i\in\text{Gal}(H_K^{\mathfrak{m}}(\zeta_p)/K(\zeta_p))\). Since \(K\) is totally real, \(H_K^{\mathfrak{m}}\) is totally real. Hence there exists \(1\neq J\in\text{Gal}(H_K^{\mathfrak{m}}(\zeta_p)/H_K^{\mathfrak{m}}(\zeta_p+\zeta_p^{-1}))\) such that \(J(\zeta_p)=\zeta_p^{-1}\). Then
\[
(J\tilde{\sigma}_i)|_{K(\zeta_p)}=J|_{K(\zeta_p)}\neq \id,~~(J\tilde{\sigma}_i)|_{K(\sqrt{E^{\mathfrak{m}}})}\neq \id,~~(J\tilde{\sigma}_i)|_{H_K^{\mathfrak{m}}}=\sigma_i.
\]
Thus \(C_i\) is admissible.

We have shown that at least one of \(C_1,\dots,C_r\) is admissible. By Theorem \ref{57}, the generating set \(\mathcal{C}\) is a Euclidean system.
\\
\end{proof}

\bibliographystyle{plain}
\bibliography{refer}

\end{document}